\documentclass[12pt]{article}
\usepackage{e-jc}
\usepackage{amssymb}
\usepackage{amsthm}
\usepackage{amsmath}
\usepackage{txfonts}

\usepackage{graphicx}
\usepackage{subfig}
\usepackage{caption}
\usepackage{sidecap}
\usepackage{wrapfig}
\usepackage{float}

\usepackage[colorlinks=true,citecolor=black,linkcolor=black,urlcolor=blue]{hyperref}

\newcommand{\arxiv}[1]{\href{http://arxiv.org/abs/#1}{\texttt{arXiv:#1}}}

\renewcommand{\labelenumi}{{\normalfont (\roman{enumi})}}

\theoremstyle{plain}
\newtheorem{theorem}{Theorem}
\newtheorem{lemma}[theorem]{Lemma}
\newtheorem*{lemma*}{Lemma}
\newtheorem{corollary}[theorem]{Corollary}

\theoremstyle{definition}

\newtheorem{conjecture}[theorem]{Conjecture}

\theoremstyle{remark}

\title{Generalised Knight's Tours}
\author{Nina Kam\v{c}ev\\
University of Cambridge\\
nk370@cam.ac.uk}
\date{23 August 2013}

\begin{document}
    \maketitle
\begin{abstract}
The problem of existence of closed knight's tours in $[n]^d$, where $[n]=\{0, 1, 2, \dots, n-1\}$, was recently solved by Erde, Gol\'{e}nia, and Gol\'{e}nia. They raised the same question for a generalised, $(a, b)$ knight, which is allowed to move along any two axes of $[n]^d$ by $a$ and $b$ unit lengths respectively.

Given an even number $a$, we show that the $[n]^d$ grid admits an $(a, 1)$ knight's tour for sufficiently large even side length $n$.

\end{abstract}

\section{Introduction}
        A \emph{knight's graph} on $[n]^2$ is the graph with the specified vertex set, whose edges are legal knight's moves (two unit lengths in one direction and one in the other). Unconventionally, throughout the paper we take $[n]=\{0, 1, 2, \dots, n-1\}$.
    A Hamiltonian cycle in a knight's graph is called a \emph{knight's tour}.
    
    The existence of a knight's tour in the $8 \times 8$ chessboard is a classical problem. The earliest known solutions originate from Arab chess players from around AD 800. The modern study of the problem appears to have begun in the late 17\textsuperscript{th} century, with \emph{R\'{e}cr\'{e}ations Math\'{e}matiques et Physiques} by Jacques Ozanam. This compilation contains knight's tours by de Montmort, de Moivre and de Mairan. For historical notes and an extensive list of papers published on special cases and related problems, we refer the reader to \cite{jelliss}.
		
    The first characterisation of rectangular chessboards which admit a knight's tour was  given by Schwenk \cite{schwenk} in 1991. 
    
    \begin{theorem}[Schwenk] \label{schwenk}
    Let $1 \leq m \leq n$. The board $[m]\times[n]$ has a closed knight's tour if and only if the following assumptions hold:
    \vspace{-5pt}
    	\vspace{-5pt} \begin{enumerate} \itemsep-1pt
    		\item At least one of $m$ and $n$ is even,
    		\item $m = 3$ or $m \geq 5$,
    		\item if $m =3$, then $n \notin \{4, 6, 8\}$.
    	\end{enumerate}
    \end{theorem}
    
    The theorem was generalised into higher dimensions in 2012. We only state the special case relevant to this paper.
    \begin{theorem}[Erde, Gol\'{e}nia, Gol\'{e}nia \cite{erde}] \label{egg}
    The grid $[n]^d$, with $d \geq 3$ admits a knight's tour if and only if $n$ is even and $n \geq 4$.
    \end{theorem}
    
    The authors of \cite{erde} also asked whether an analogous knight's tour exists if the standard knight is replaced with the generalised, $(a, b)$ knight. The \emph{$(a, b)$ knight} is allowed to move along any two axes of $[n]^d$ by $a$ and $b$ unit lengths respectively. We define the \emph{$(a, b)$ knight's graph} on vertex set $[n]^d$ to be the graph whose edges are legal moves of the $(a, b)$ knight.
		
    We refer to an \emph{$(a, b)$ knight's tour in $[n]^d$} to mean a Hamiltonian cycle in the corresponding knight's graph. Graph theory terminology is used throughout the paper.
        
    The first generalisation of Theorem \ref{egg} is the existence of an $(a, 1)$ knight's tour, for any even  $a$ (for odd $a$ this is plainly impossible since the graph has two connected components: $\{(x_1,\,x_2,\, \dots x_d): \sum_i{x_i} \text{ odd}\}$ and its complement). Our main task is to find a 2-dimensional $(a, 1)$ knight's tour. It turns out that extending to $d$ dimensions is comparatively easy.
		
        The following result is proved in Section 2.
    \begin{theorem} \label{thm1}
    For any even values of $a$ and $n$, with $n \geq a(6a+2)$, there exists an $(a, 1)$ knight's tour in $[n]^d$.
    \end{theorem}
         
        For odd $n$ and even $a$, no knight's tour exists because the $(a, 1)$ knight's graph is bipartite, with two partitions of different size.
        
        There are two main points in which our 2-dimensional $(a, 1)$ knight's tours differ from the ones constructed so far for the $(2, 1)$ knight. Firstly, a natural way to build a Hamiltonian cycle is to assemble it from a sequence of paths with adjacent endpoints, returning to the initial vertex. It requires much less caution and work to just partition the vertices of the graph, find a cycle in each partition, and then gradually merge (or \emph{concatenate}) these into a single Hamiltonian cycle, which is what is done in Section 2.
        
        Secondly, our vertex-partitions produce cycles that extend through the whole board, rather than being localised to square or rectangular domains. This `globality' property is apparent in the resulting $(a, 1)$ knight's tours.
        
        By analogy with the $(a, 1)$ knight, in Section 3 we give a way of extending $(a, b)$ knight's tours in $[n]^2$ (with very little additional structure) to knight's tours in $[n]^d$. Note that $a$ and $b$ are required to be coprime and not both odd if we are to construct an $(a, b)$ knight's tour. We also find $(2, 3)$ and $(2, 5)$ knight's tours in $[n]^2$ for sufficiently large even $n$, and consequently in $[n]^d$. It is much harder to find structure in these graphs. In particular, we use completley different methods for the $(2, 5)$ knight's tour (Section 3.3). We do not know what happens for general $(a, b)$ knights.
				
			In Section 2, definitions and results are mostly stated in terms of coordinates, but we use diagrams to illustrate them and encourage the reader to think in terms of diagrams.Throughout Section 2, $C_{ij}$ means $C_{i, j}$ and similar. We omit commas in subscripts to make the text and diagrams more readable unless there is potential ambiguity. All numerical values are one-digit.

\section{The \texorpdfstring{$(a, 1)$}{(a, 1)} knight}
	The construction of the $(a, 1)$ knight's tour in $[n]^2$ comprises most of this section. We first find two ways of extending existing knight's tours, which will allow us to induct on the side length $n$. Then we turn to the basic case, $n=6a+2$.
	\subsection{Extending 2-dimensional \texorpdfstring{$(a, 1)$}{a, 1} knight's tours}
				Concatenation of cycles is a basic principle for building longer cycles from the given, smaller ones.
			\begin{center}
				\includegraphics[scale=0.7]{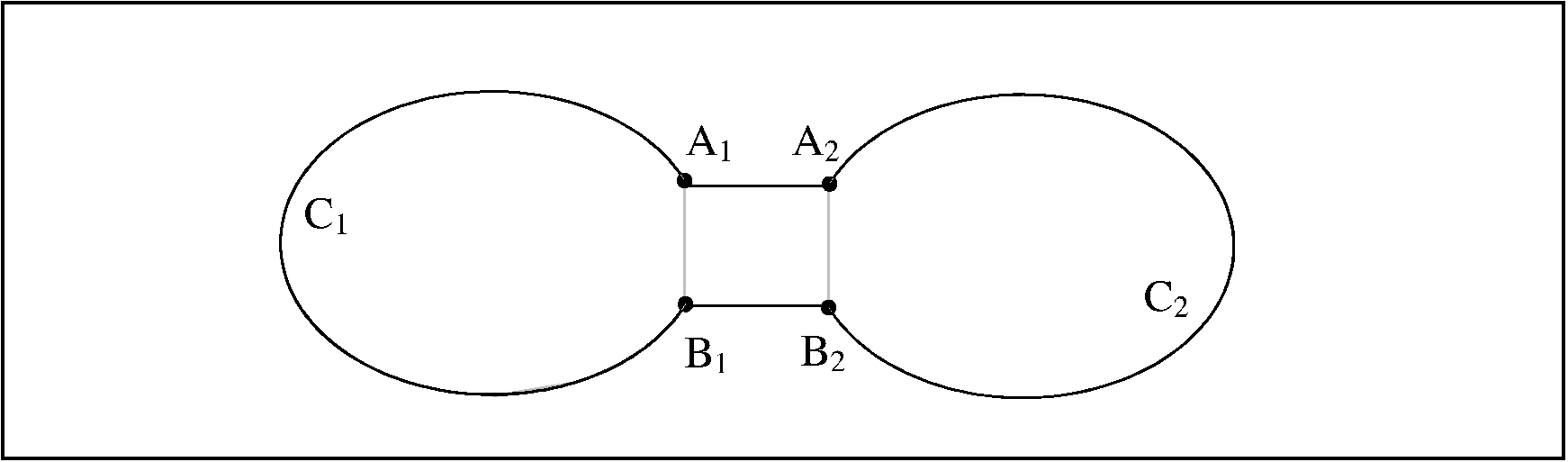}
			\end{center}
				
				Consider a graph $G$ containing vertex-disjoint cycles $C_1$ and $C_2$. Let $C_1$ and $C_2$ contain edges $A_1 B_1$ and $A_2 B_2$ respectively. Suppose $A_1$ and $A_2$ are adjacent, and $B_1$ and $B_2$ are adjacent in $G$. Then the following is a cycle containing the union of vertices of $C_1$ and $C_2$:
$$A_1  \stackrel{C_1}{\longrightarrow} B_1 - B_2   \stackrel{C_2}{\longrightarrow} A_2.$$
				We call this the \emph{compound} of cycles $C_1$ and $C_2$ and denote it by $C = C_1 \cup C_2$. The edges $A_1 B_1$ and $A_2 B_2$ will be called a \emph{bridge} for this compound.
			  
		\subsubsection{Traversing the rim of the board}
		
				Consider an $(a, 1)$ knight's graph with vertex set $[n]^2$, where $a$ and $n$ are even. In our diagrams, $(0, 0)$ is the bottom left square, so that the coordinatisation is in accordance with the first quadrant of the Cartesian coordinate system. We call the subgraph induced by $M = \{ (i, j) \mid a \leq i \leq n-a-1,\, a \leq j \leq n-a-1 \}$ the \emph{middle} of the board. The subgraph induced by the complement of $M$ is called the \emph{rim} of width $a$.
				
				A \emph{unit} is defined as a set of form $\{(i,\, j),\, (i+1,\, j),\, (i,\, j+1),\,(i+1,\, j+1)\}$, where $i$ and $j$ are even, as shown in Figure \ref{g2}. Two units $U$ and $U^\prime$ are \emph{equivalent} if $U$ can be translated by $(ka, la)$ ($k$ and $l$ integers) to cover $U^\prime$. For example, the highlighted units in Figure \ref{g2} are an equivalence class of units in the rim of width 4 of our (4, 1) knight's graph.
								
				We $\emph{colour}$ the vertices of the rim using $a^2$ colours, as displayed in Figure \ref{g2}. The set \emph{coloured} by $i, j$ is uniquely defined by the requirement that its vertices lie within equivalent units. Each colour $i, j$ induces a cycle $C_{ij}$ in the rim of width $a$.
							
			\begin{figure}[!ht]
			\centering
			\includegraphics[scale=0.8]{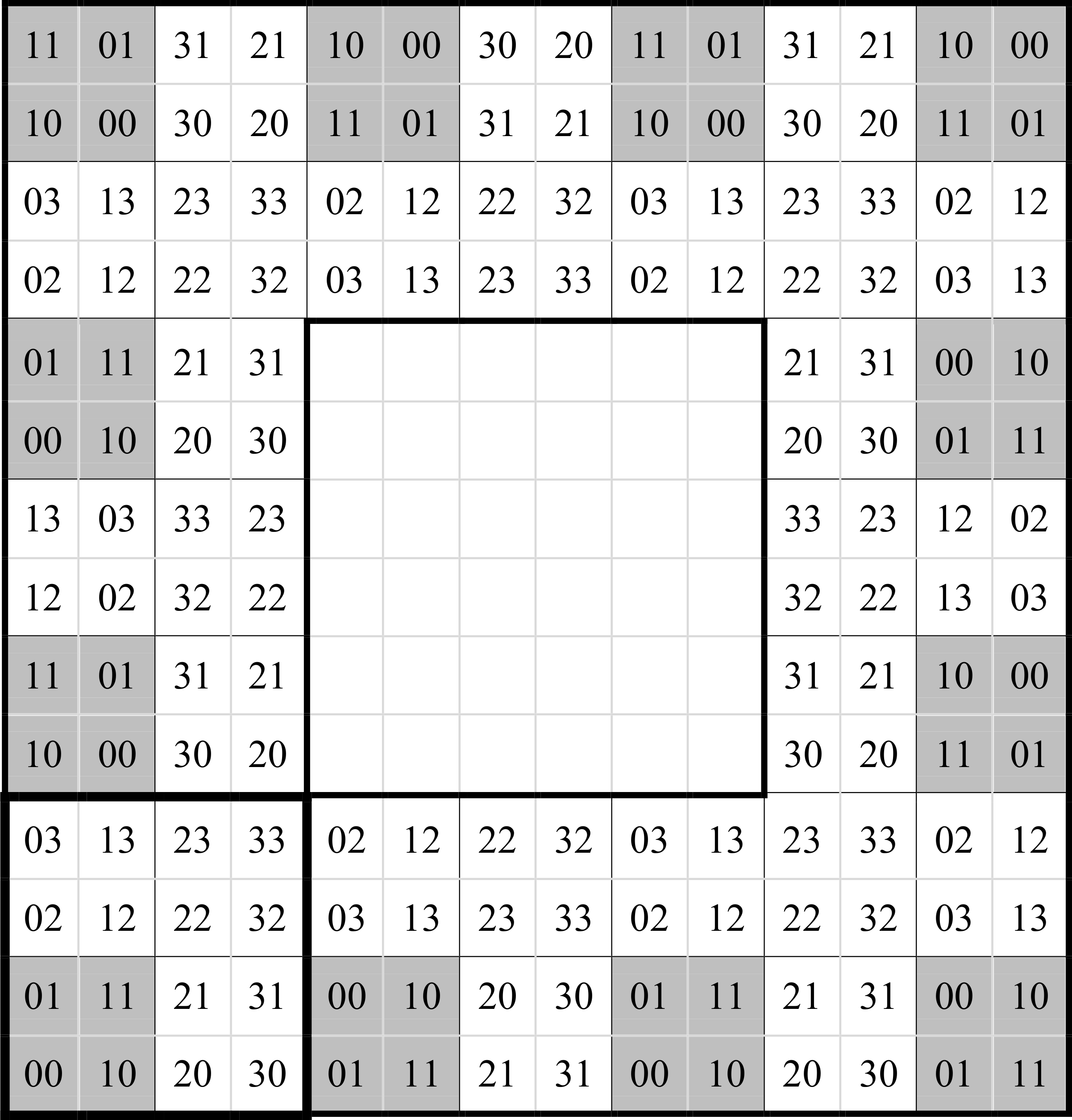}
			\caption{Colours defined by the $(4, 1)$ knight in the rim of width 4 of $[14]\times [14]$. Each $4 \times 4$ square contains exactly one vertex of each colour.}  \label{g2}
		\end{figure}
		
		To explain why colours indeed induce closed paths $C_{ij}$ for any even $n$, it is sufficient to look at $C_{00},\, C_{01},\, C_{10},\, C_{11}$, and track their position within the highlighted (gray) units. The other cycles will behave analogously.
		
		The vertices belonging to $C_{00}$ and $C_{01}$ alternate vertically in bottom and top segment of the rim. The number of these transpositions is the same at the bottom and at the top, which means that $C_{00}$ and $C_{01}$ end up in the same relative position after a single lap ($C_{00}$ below $C_{01}$). Similarly, vertices belonging to $C_{00}$ and $C_{10}$ alternate horizontally in right and left segment of the rim, ending up at the same relative position. So cycle $C_{00}$ indeed returns to the bottom left corner.
		 
		Equivalently, $C_{01}$ alternates with $C_{00}$ vertically and with $C_{11}$ horizontally to return to the top left corner of its unit. The same argument works for $C_{10}$ and $C_{11}$, so the claim is proved.

				The partition of the rim into colours is not a matter of choice - most vertices along the edge of the board have degree 2 in the rim, as well as most vertices along the middle. These constraints dictate the 'propagation' of cycles.
			 	
				To be able to induct on the side length, we prove a slightly stronger claim than mere existence of a knight's tour. To formulate the hypothesis, we define a \emph{structured} $(a,1)$ knight's tour as a knight's tour that contains edges
				\vspace{-6pt}
				\begin{itemize} \itemsep-1pt
					\item $\{(i, j), (i+a, j+1) \}$,   where $i$ and $j$ lie between $0$ and $a-1$, and $j$ takes only even values,
					\item $\{(i, j), (i+a, j-1) \}$,   where $i$ and $j$ lie between $0$ and $a-1$, and $j$ takes only odd values.
				\end{itemize}			  
			\begin{figure}[!ht]
				\centering
				\includegraphics[width=0.7\textwidth]{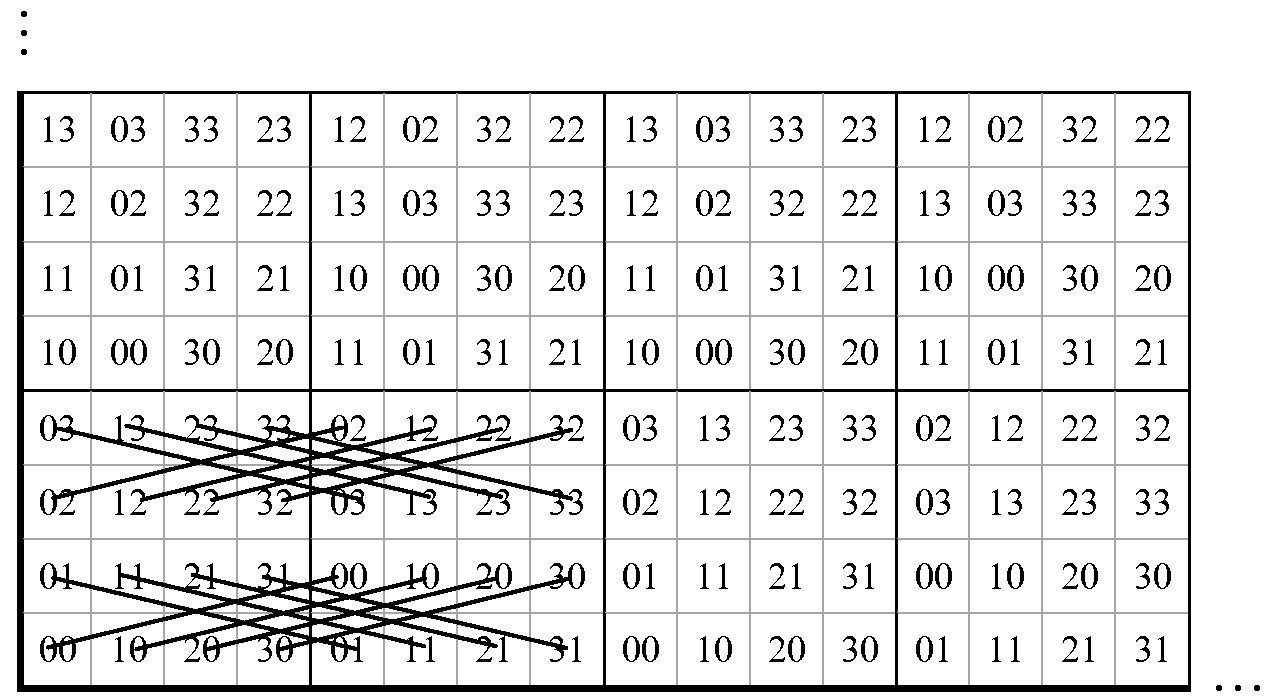}
				\caption{Bottom left corner of a $(4, 1)$ knight's graph, with edges required for a structured knight's tour highlighted.}  \label{g3a}
			\end{figure}
				We shall contract the two cases as $\{(i, j), (i+a, j+(-1)^j)\}$. Each of these edges (shown in Figure \ref{g3a}) is contained in the corresponding cycle $C_{ij}$ so this requirement is not very strong.
			
			\begin{lemma} \label{l1} 	
 				Let a structured $(a, 1)$ knight's tour in $[n]^2$ be given, with $a$ and $n$ even and $n \geq 2a$. Then there exists a structured $(a, 1)$ knight's tour in $[n+2a]^2$.
 			\end{lemma}
 			\begin{proof}
				The middle of the board $[n+2a]^2$ is isomorphic to the $(a, 1)$ knight's graph in $[n]^2$ in a natural way, so we can construct a structured $(a, 1)$ knight's tour $C$ in $\{a, a+1, \dots \ n+a-1\}^2$. By definition, $C$ contains the edges $\{(i+a, j+a), (i+2a, j+(-1)^j+a) \}$, for all $i$ and $j$ in $[a]$. These edges, for $j \in \{0, 1\}$ are represented by solid black edges in Figure \ref{g3}.
 				
				Also form colours $C_{00}$ to $C_{a-1,a-1}$ in the rim of width $a$. These cycles contain the edges $\{(i+a+1, j), (i+2a+1, j+(-1)^j) \}$ for $i$ and $j$ in $[a]$ (represented by dashed black edges in Figure \ref{g3} for $j \in \{0, 1\}$).
				
 				As explained in the previous section, for all $i$ and $j$ in $[a]$, $\{(i+a+1, j), (i+2a+1, j+(-1)^j) \}$ and $\{(i+a, j+a), (i+2a, j+(-1)^j+a) \}$ is a bridge for the compound of $C_{i+1, j}$ and $C$ (we take $C_{a, j}= C_{0, j}$). So by replacing this pair with $\{(i+a+1, j), (i+a, j+a) \}$ and $\{(i+2a+1, j+(-1)^j), (i+2a, j+(-1)^j+a) \}$ (gray edges in Figure \ref{g3}), we get the compound
$$\left( \bigcup_{i, j \, \in [a]} C_{ij} \right) \cup C. $$
			
			\begin{figure}[!ht]
				\centering
				\includegraphics[width=0.5 \textwidth]{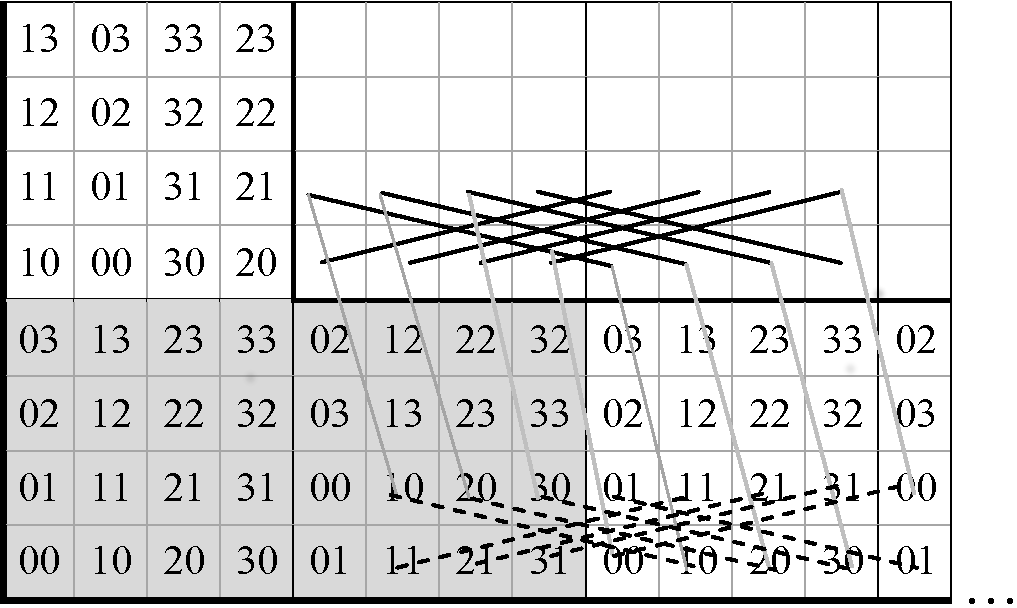}
				\caption{Bottom left corner of the $(4, 1)$ knight's graph on $[n+2a]^2$. The black edges are bridges between the middle (cycle $C$) and the rim (cycles $C_{ij}$ with $i=0, \dots a-1$ and  $j=0, 1$). They are replaced by gray edges to form our new $(a, 1)$ knight's tour.}  \label{g3}
			\end{figure}

This notation makes sense because once the cycles are concatenated, the order in which it has been done is not distinguishable. The result is an $(a, 1)$ knight's tour in $[n+2a]^2$. Compare this new knight's tour with the definition of a structured tour. Edges connecting the two highlighted $4 \times 4$ squares are not among the ones which have been deleted to do the concatenation in Figure \ref{g3}. This means exactly that the $(a, 1)$ knight's tour we constructed is structured, which completes the proof. 
 			\end{proof}
 			
 		\subsubsection{Sequential concatenation of \texorpdfstring{$(a, 1)$}{a, 1} knight's tours}
			Given a collection of vertex-disjoint cycles $\{C_i\}$ in a knight's graph, the most straightforward way to concatenate all of them into a single cycle is to order them in a sequence, and then concatenate $C_i$ and $C_{i+1}$ for each $i$. In addition, Lemma \ref{l2} gives a strategy of concatenating $C_i$ and $C_{i+1}$ in a way that uses only one link, which will be used in many seemingly different constructions. This is a crucial idea, and we will refer to it as \emph{sequential concatenation}.
			
			We call the edge $\{(a-1, 1), (a-2, a+1)\}$ \emph{link A} and refer to an $(a, 1)$ knight's tour containing it as an \emph{A-linked} knight's tour.
			
			\begin{lemma} \label{l2}
				Let $[n]^2$ admit a structured $(a, 1)$ knight's tour and an A-linked $(a, 1)$ knight's tour. Then for any natural number $k$, there is a structured $(a, 1)$ knight's tour in $[kn]^2$.
			\end{lemma}
			\begin{proof}
				The hypothesis of this lemma implies that $a$ and $n$ are both even. Consider first just the board $[2n] \times [n]$, viewed as two copies of $[n]^2$. Endow each copy of $[n^2]$ with an $(a, 1)$ knight's tour. The right copy, $\{n, n+1, \dots 2n-1\} \times [n]$ contains link A by assumption, which is given by $\{(n+a-1, 1), (n+a-2, a+1)\}$ in the new common coordinate system. Then the edges \\ $\{(n+a-1, 1), (n+a-2, a+1)\}$ (link A) and \\ $\{(n-1, 0), (n-2, a)\}$ (which must exist as a corner edge) make up a bridge. We use this bridge to construct a single $(a, 1)$ knight's tour $[2n] \times [n]$.
			Represent this by an arrow pointing into the right copy, indicating that its link A has been 'used up' for the concatenation:
			\begin{center}
				\includegraphics[scale=0.8]{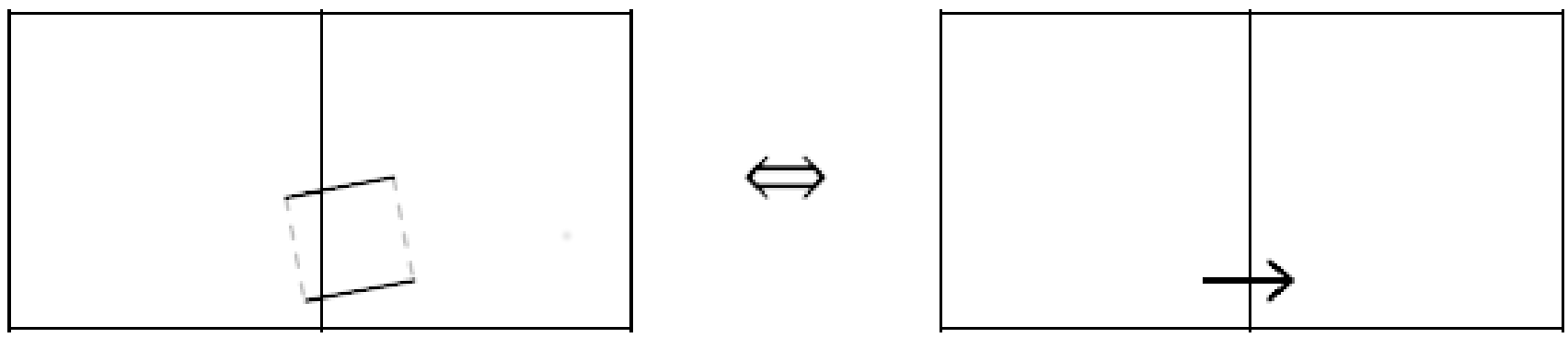}
			\end{center}
			
				Now extend this concept to $[kn]^2$, divided into $k^2$ copies of $[n]^2$, called \emph{elementary boards}. Regard the elementary boards as vertices of a $k \times k$ grid graph. In this grid graph, two vertices are adjacent whenever they share a side. The $k \times k$ grid graph admits a straightforward Hamiltonian path $P$ along its rows, which we indicate by arrows (Fig.~\ref{g9}).				
				\begin{figure}[!htb]
					\centering \includegraphics[scale=0.4]{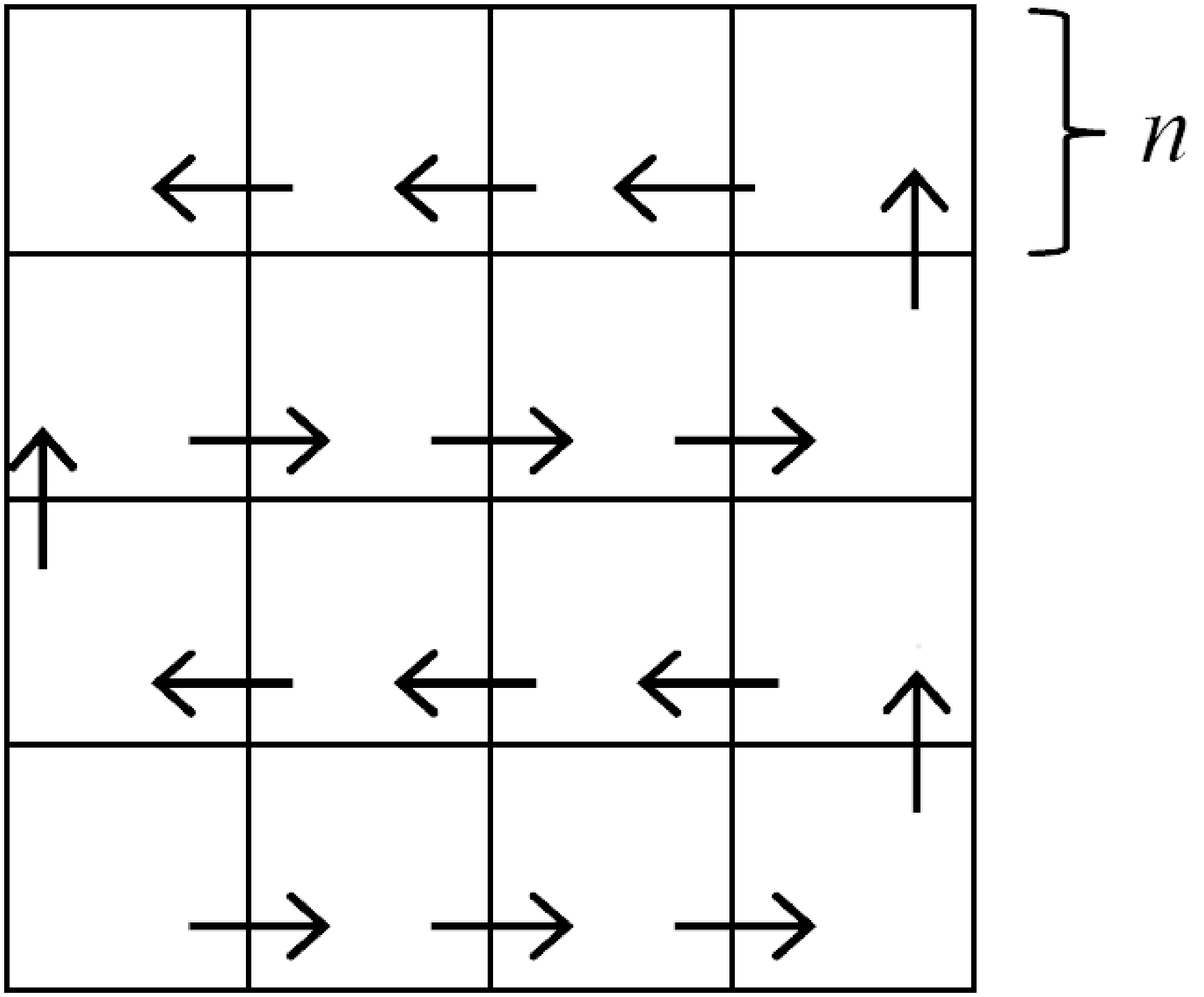}
					\caption{Grid graph whose vertices are elementary boards (for $k=4$). The arrows indicate the Hamiltonian path $P$ used for sequential concatenation.} \label{g9}
				\end{figure}
				By assumption, each elementary board admits an A-linked knight's tour. Flip and rotate these tours so that each link A is positioned in accordance with the corresponding arrow. Now we can perform concatenations dictated by the arrows, one by one, as in Figure \ref{g9}.
				
				Link A of the $(a, 1)$ knight's tour in each elementary board is used up exactly once. The exception is the bottom left one, $[n]^2$, for which we haven't made any assumptions. We take the initial $(a, 1)$ knight's tour in this board to be structured, so that this structure is inherited by the new tour in $[kn]^2$.
			\end{proof}

			\newpage
				It will turn out that Lemma \ref{l1} and Lemma \ref{l2} are the key to extending a \emph{suitable} basic case to a knight's tour in any sufficiently large board.
						
	\subsection{Basic case}
			We will build all our $(a, 1)$ knight's tours from a structured $(a, 1)$ knight's tour in $[6a+2]^2$. The construction is fairly complex. Namely, it is somewhat surprising that such an $(a, 1)$ knight's tour can be built for general (even) $a$.

			We now partition the vertices of the entire board, as we partitoned the rim in the previous section. We extend the definition of a \emph{unit} (a set of form $\{(i,\, j),\, (i+1,\, j),\, (i,\, j+1),\,(i+1,\, j+1)\}$, where $i$ and $j$ are even) to the entire vertex set $[6a+2]^2$. Recall, two units $U$ and $U^\prime$ are \emph{equivalent} if $U$ can be translated by $(ka, la)$ ($k$ and $l$ integers) to cover $U^\prime$.	The set of vertices numbered by $i, j$ will now be called a \emph{level} and denoted by $L_{ij}$. We construct the level $L_{ij}$ as follows (see Figure \ref{g4}).
			
			\begin{figure}[!ht]
				\centering
				\includegraphics[scale=0.74]{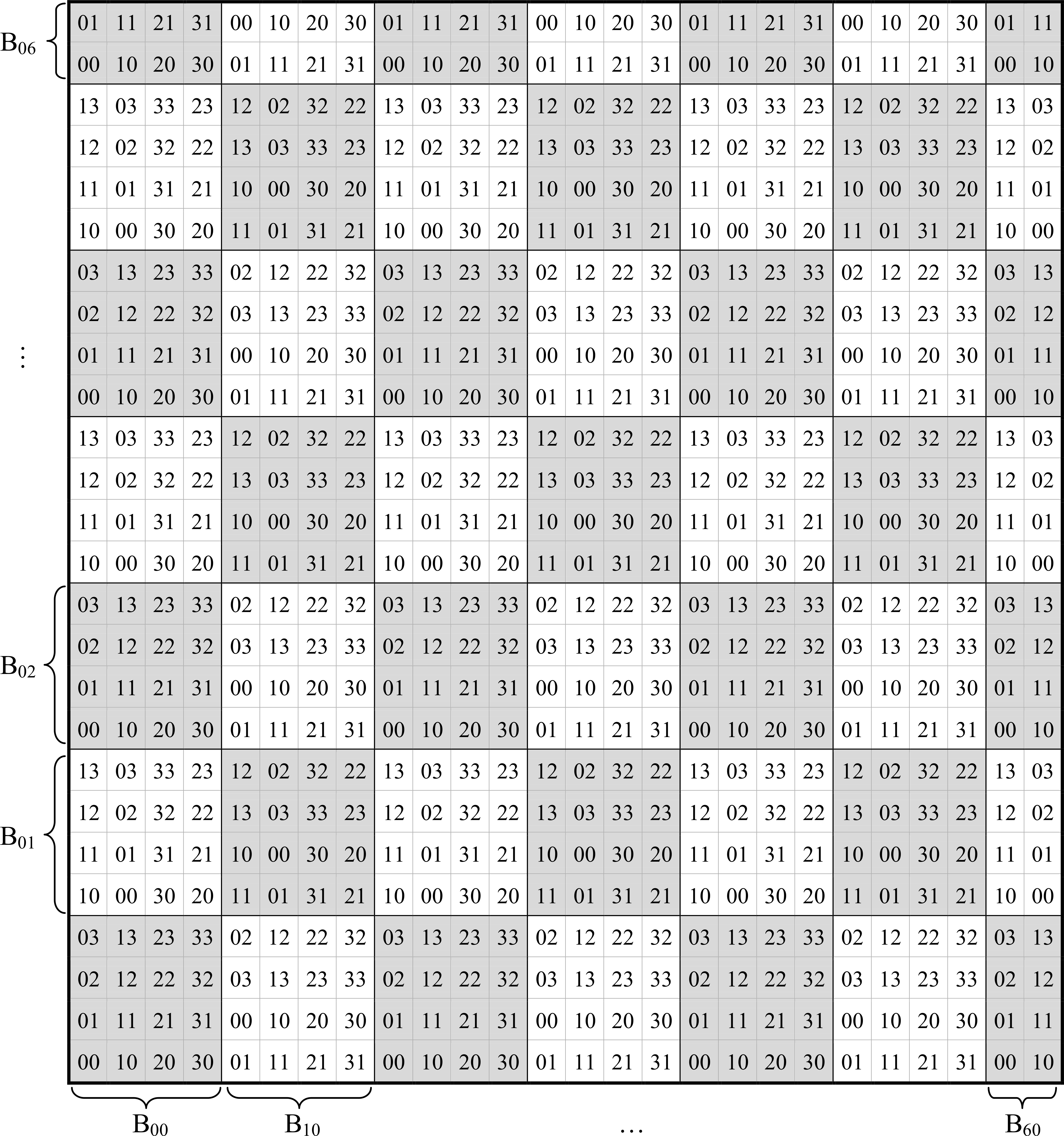}
				\caption{The levels on $[6a+2]^2$ (here $a=4$). Division into blocks is indicated by gray shading.
} \label{g4}
			\end{figure}

			\vspace{-5pt} \begin{enumerate} \itemsep-1pt
				\item $L_{ij}$ contains the vertex $(i, j)$.
				\item Consider the subgraph of our knight's graph induced by the unit containing $(i, j)$ and all equivalent units. $L_{ij}$ is the connected component of this subgraph determined by (i).
			\end{enumerate}
\vspace{-5pt}
			We refer to two levels $L$ and $L ^{\prime}$ as \emph{adjacent} if there are vertices belonging to $L$ and $L^\prime$ respectively that are related by a single knight's move.
In particular, if the knight's move that relates them is $(-1, \pm a)$ or $(1, \pm a)$ then we say $L$ and $L^\prime$ are \emph{vertically adjacent}. Otherwise we say that they are \emph{horizontally adjacent}. Refer to Figure \ref{g4} to make sure these concepts are well-defined. Lemma \ref{l2a} shows that each level is adjacent to four others - two horizontally and two vertically.

			Each set of vertices $B_{ij}=\{(x, y) \mid ia \leq x \leq (i+1)a-1,\,\, ja \leq y \leq (j+1)a-1\}$ is called a \emph{block} ($i$ and $j$ range from 0 to 6). We refer to blocks $B_{i,6}$ and $B_{6,j}$ along the top and right edge of the board as \emph{incomplete} blocks. This allows us to specify the knight's position using ordered pairs $(B_{i_1, j_1}, L_{i_2, j_2})$ - it is essentially a new coordinatisation of $[6a+2]^2$.
			
			The moves of an $(a, 1)$ knight constrained to one level are now reduced to just horizontal and vertical steps from one block to its neighbour (neighbours are blocks which share a side). This means that we can view each level separately as a 2-dimensional grid graph. Keeping in mind that movement within each level is simple, our strategy is to cover each level entirely before switching (or \emph{lifting}) to the next one. But reaching each level exactly once and returning to the starting point is yet another Hamiltonian-cycle problem. Hence we view our levels as vertices of a new graph. Following its Hamiltonian cycle will enable us to lift between levels.
			
			\subsubsection{The guide}
				 
					Let $G$ be a graph with vertices denoted by $L_{ij}$, where $i$ and $j$ range from $0$ to $a-1$ and $a$ is even. Two vertices of this graph are adjacent if the corresponding levels are adjacent. We read off the adjacent vertices in this graph from Figure \ref{g4} and its analogue, Figure \ref{g4n_72}. Lemma \ref{l2a} is followed by a restatement which you might find easier to picture.
					
				\begin{lemma} \label{l2a} \vspace{-3pt}
					Let $G$ be the graph defined above, for some even $a$. The edges of $G$ are as follows. 
					\vspace{-5pt} \begin{enumerate} \itemsep-1pt
						\item If $i \geq 2$, then $L_{ij}$ is vertically adjacent to $L_{i-2, j}$.\\
									If $i \leq a-3$, then $L_{ij}$ is vertically adjacent to $L_{i+2, j}$.
						\item If $i \in \{0, 1\}$, then $L_{ij}$ is vertically adjacent to $L_{i-2+a, j+(-1)^j}$.\\
									If $i \in \{a-2, a-1\}$, then $L_{ij}$ is vertically adjacent to $L_{i+2-a, j+(-1)^j}$.
						\item If $j \geq 2$, then $L_{ij}$ is horizontally adjacent to $L_{i, j-2}$.\\
									If $j \leq a-3$, then $L_{ij}$ is horizontally adjacent to $L_{i, j+2}$.
						\item If $j \in \{0, 1\}$, then $L_{ij}$ is horizontally adjacent to $L_{i+(-1)^i, j-2+a}$.\\
									If $j \in \{a-2, a-1\}$, then $L_{ij}$ is horizontally adjacent to $L_{i+(-1)^i, j+2-a}$.
					\end{enumerate}
				
				\end{lemma}
								
			\begin{figure}[!hb]
				\centering
				\includegraphics[scale=0.7]{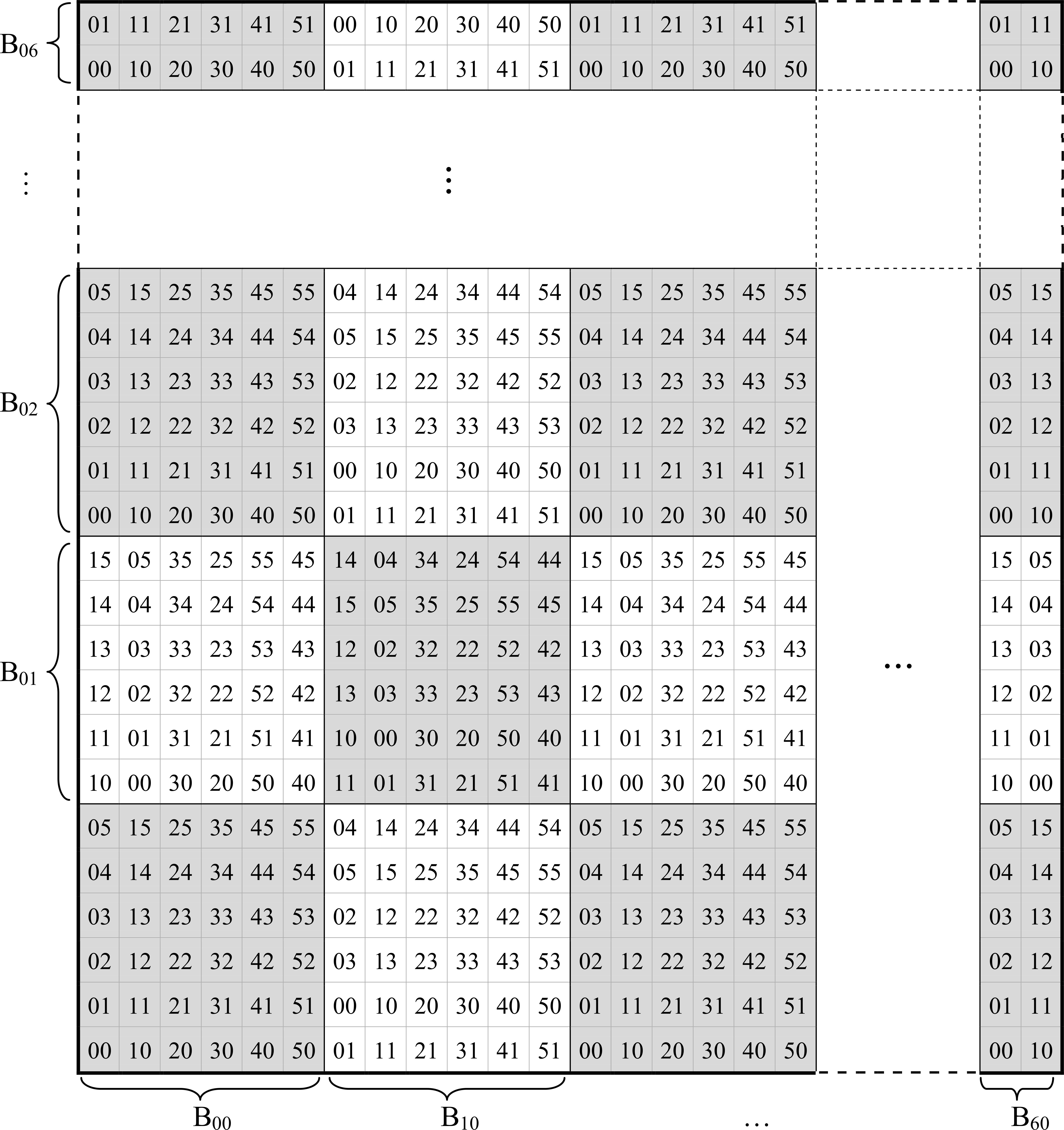}
				\caption{The levels on $[6a+2]^2$ (here $a=6$). Only blocks relevant to the proof of Lemma \ref{l2a} are shown. Division into blocks is indicated by gray shading. 
} \label{g4n_72}
			\end{figure}
					
				\begin{proof}
					For (i) and (ii), we look at blocks $B_{11}$ and $B_{12}$, shown in Figure \ref{g4n_72} for $a=6$. Statements (iii) and (iv) follow from symmetry across the main diagonal.
					
					Let $i$ be even, and look at block $B_{11}$. The moves $(1, -a)$ and $(1, a)$ take us from $(B_{11}, L_{ij})$ to $(B_{10}, L_{i+2, j})$ and $(B_{12}, L_{i+2, j})$ respectively, provided that $i \neq a-2$. In this case $L_{ij}$ and $L_{i+2, j}$ are vertically adjacent.
					
					The same moves connect $(B_{11}, L_{a-2, j})$ to $(B_{20}, L_{0, j+(-1)^j})$ and $(B_{22}, L_{0, j+(-1)^j})$. Note the diagonal step from block $B_{11}$ to $B_{20}$ or $B_{22}$.
					
					In the argument above, we can replace $B_{11}$ by any $B_{kl}$ with $l$ odd. This will be used in the proof of lemma \ref{l4}.
					
					Looking at $B_{12}$ gives us that $(B_{12}, L_{ij})$ is vertically adjacent to $(B_{1, 2 \pm 1}, L_{i-2, j})$ via $(-1, \pm a)$ unless $i=0$. Also using the move $(-1, \pm a)$, $(B_{11}, L_{0j})$ is adjacent to $(B_{0, 1 \pm 1}L_{n-2, j+(-1)^j})$.
					
					For $i$ odd, the same argument works, but we have to use the moves $(1, \pm a)$ to move from block $B_{11}$ and $(-1, \pm a)$ to move from block $B_{12}$. This completes the proof.
				
				\end{proof}
			
				We now write down the same lemma in terms of diagrams. This is used to construct the \emph{guide}, a Hamiltonian cycle in $G$. The vertices of $G$ are assigned \emph{relative unit-positions} (sometimes just \emph{positions}). Level $L_{ij}$ has position A if $i$ and $j$ are both even, B if $i$ is odd and $j$ is even, C if $i$ and $j$ are both odd, D if $i$ is even and $j$ is odd (cf. Figure \ref{g5}). The assignment yields a partition of vertices of $G$ into four subsets.

			\begin{figure}[!htb]
				\centering
				\subfloat[labelformat=simple][]{
					\includegraphics{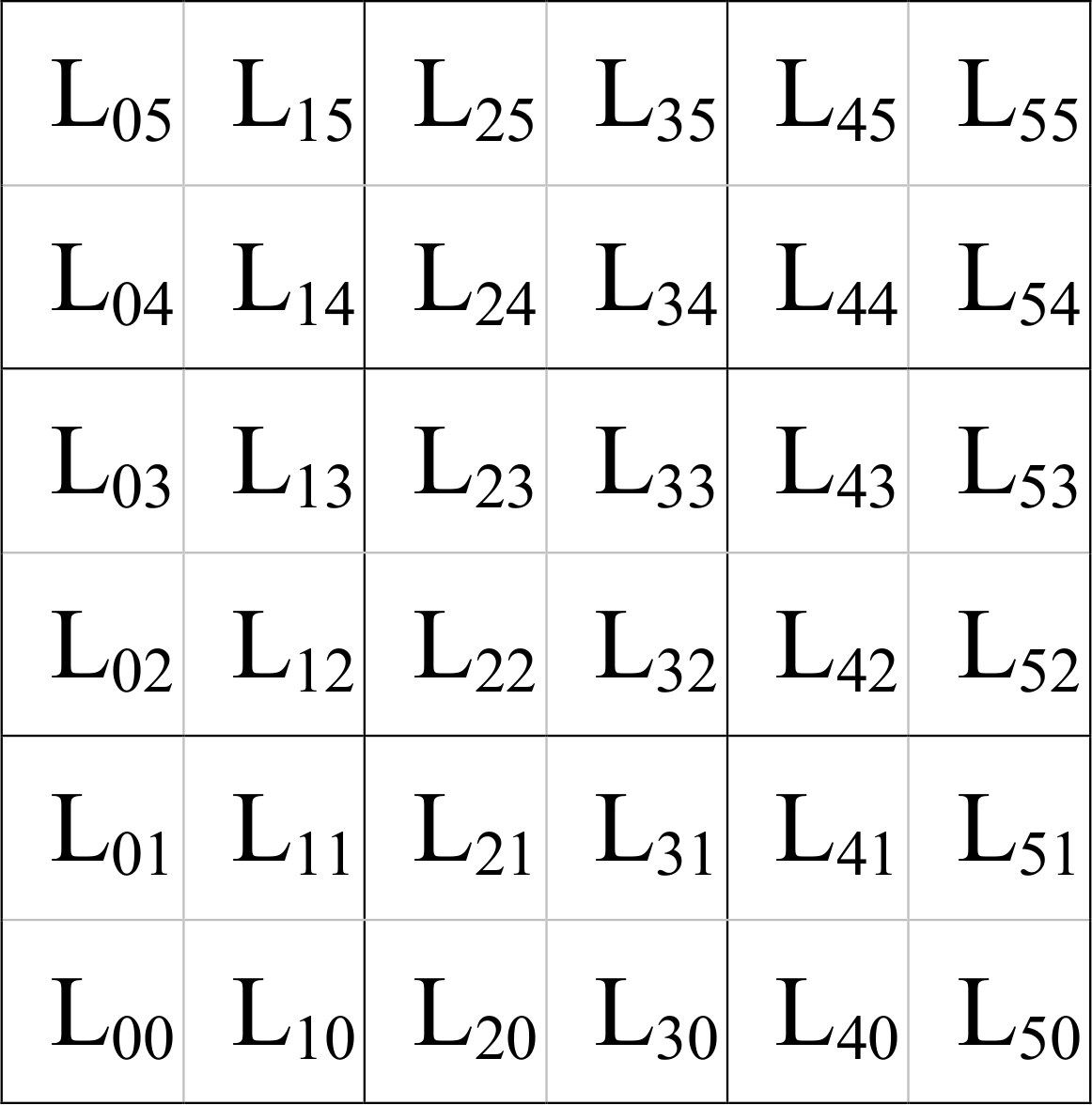}
					}
				\hspace{2cm}
				\subfloat[labelformat=simple][]{
					\includegraphics{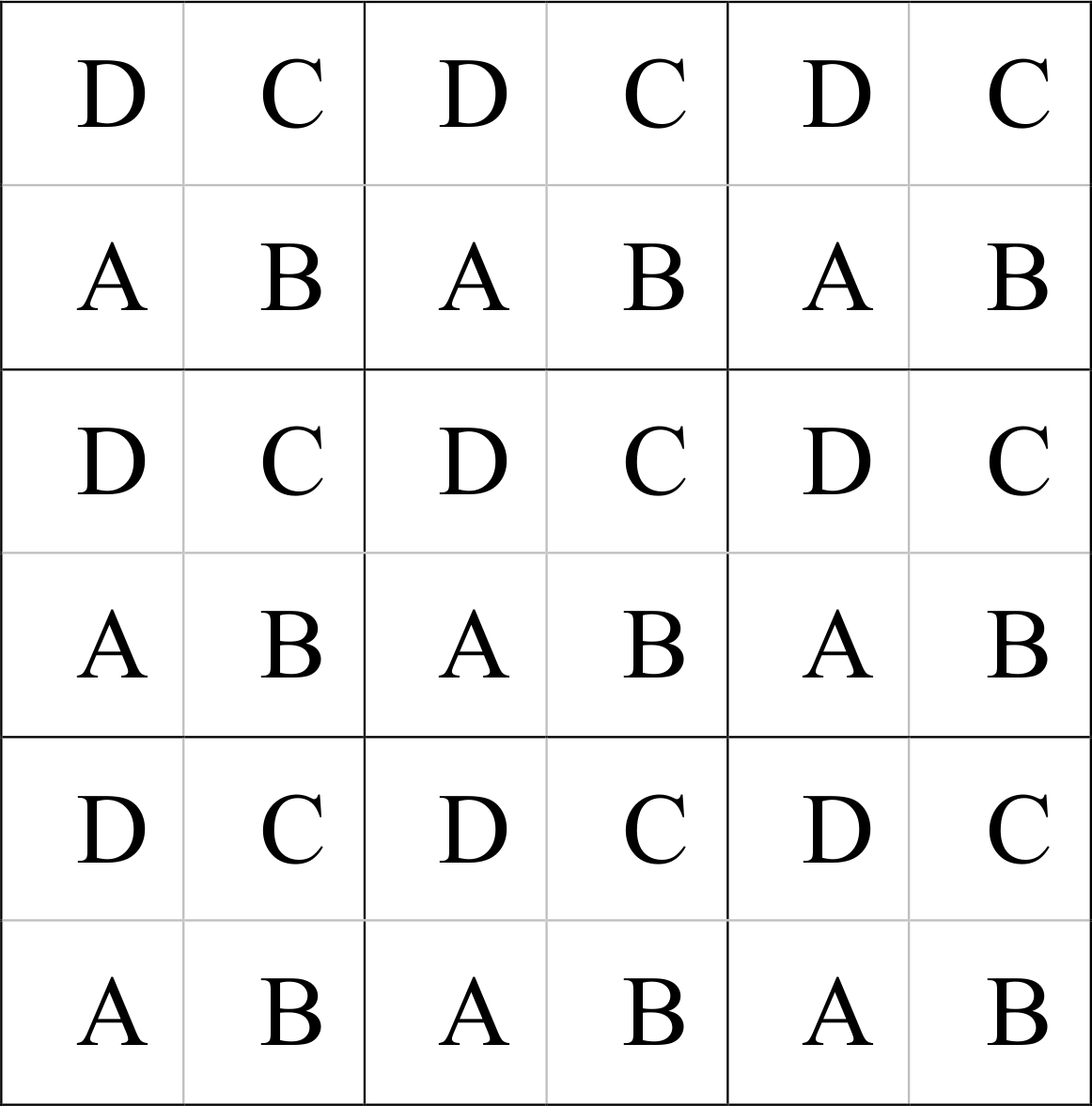}
					}
					\caption{The graph $G$ for $a=6$, divided into units. Levels are assigned positions according to the table (b).
} \label{g5}					
			\end{figure}
			
				Then the permitted moves in the guide are as follows:
				\vspace{-5pt} \begin{enumerate} \itemsep-1pt
					\item From any unit, we can move to its horizontal neighbour. In doing so, the relative unit-position is unchanged (e.g. $L_{32}$ is vertically adjacent to $L_{12}$  and $L_{52}$, by looking at $B_{11}$ and $B_{12}$).
					\item If the unit in which we are positioned is along the left edge of Figure \ref{g5}(a) (that is, we are in one of the levels $L_{0, j}$ or $L_{1, j}$, any $j$), we can step `over the edge' to level $L_{a-2, j+(-1)^j}$ or $L_{a-1, j+(-1)^j}$ respectively. In doing so, the relative unit-position changes vertically, but not horizontally (i.e. A $\leftrightarrow$ D and B $\leftrightarrow$ C). Similarly, if we are in one of the levels $L_{a-2, j}$ or $L_{a-1, j}$, any $j$, we can step to the left-hand side of G with the same effect on relative unit-position. We will call this move a \emph{special move} or \emph{special lifting}.
					\item From any unit, we can move to its vertical neighbour without changing relative unit-position (e.g. $L_{32}$ is horizontally adjacent to $L_{30}$  and $L_{34}$).
					\item If the unit in which we are positioned is along the top or bottom edge of Figure \ref{g5}(a), we can step 'over the edge', while changing relative unit-position horizontally (A $\leftrightarrow$ B and D $\leftrightarrow$ C). We call this a \emph{special move} or \emph{special lifting} as well.
				\end{enumerate}
				
				For example, level $L_{01}$ is adjacent to $L_{21}$ and $L_{03}$ via (i) and (iii), and to $L_{a-2, 0}$ and $L_{1, a-1}$ via special moves (iii) and (iv).
				
				We are now ready to construct the guide.
				
				\begin{lemma} \label{l3}
					For any even $a$, the graph $G$ admits a Hamiltonian cycle (the guide).
				\end{lemma}
				\begin{center}
					\includegraphics[scale=1]{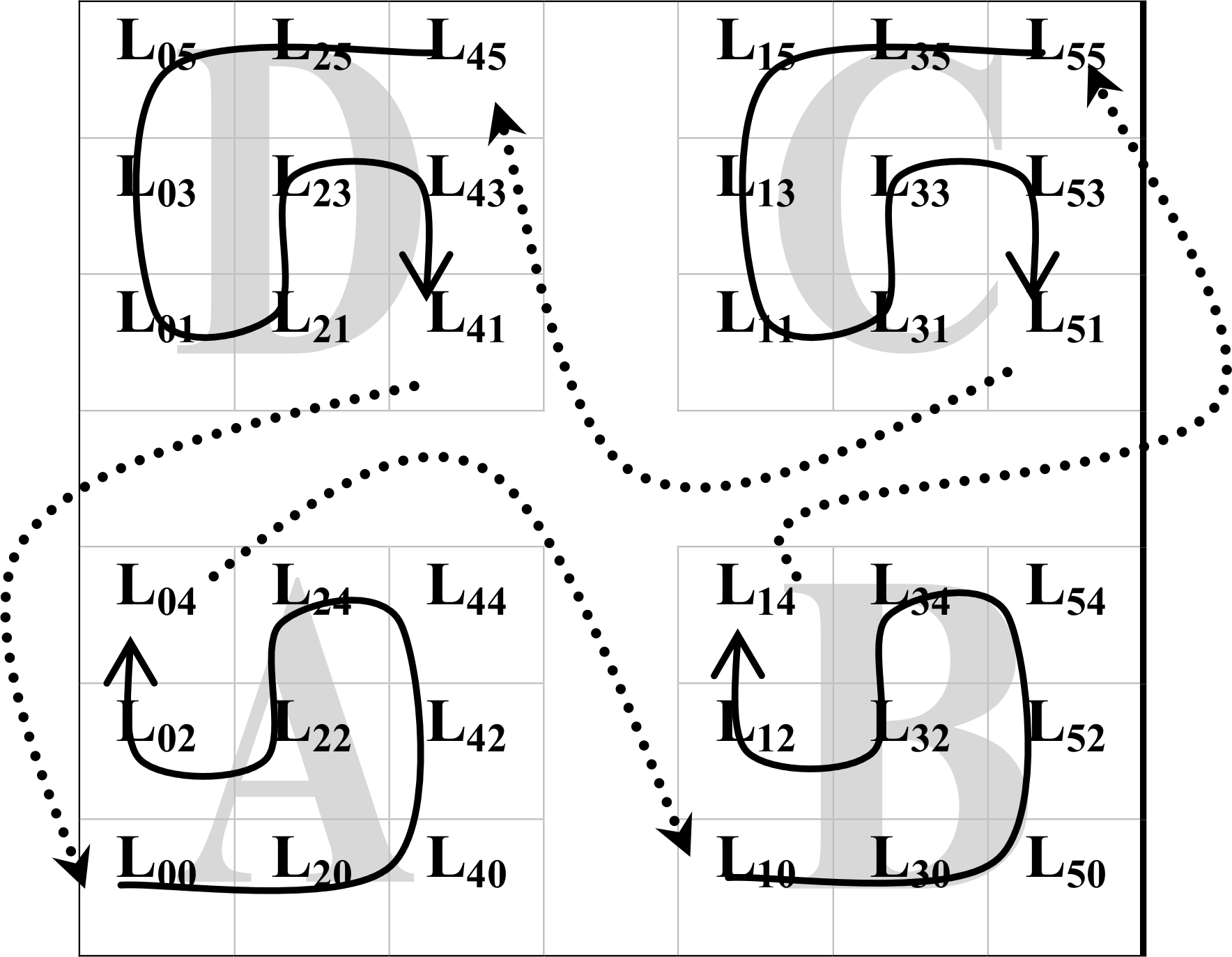}
				\end{center}
				
				\begin{proof}
					The figure above is another representation of $G$ (we use the example $a=6$). It consists of four $\frac{a}{2} \times \frac{a}{2}$ grid graphs related by special moves. Vertices of each grid graph are levels which have the same relative unit-position (this position is specified in the background).
					
Our guide starts at $L_{00}$ and follows the solid arrow to cover the entirety of position A and finish at vertex $L_{0, a-2}$. It is important to notice that such a path exists regardless of whether $\frac{a}{2}$ is odd or even.

Then use a special move (dotted arrow) from position A (vertex $L_{0, a-2}$) to position B (vertex $L_{10}$) and cover all vertices with position B. Follow the dotted arrow to vertex $L_{a-1, a-1}$ with position C. Continue with positions C and D.

We finish at vertex $L_{a-2, 1}$ with position D, which is adjacent to $L_{00}$ via a special move.
				\end{proof}
				
		\subsubsection{Implementing the guide}
				We denote blocks starting from the bottom left corner by $B_{ij}$, with $i$ and $j$ ranging from 0 to 6. Block $B_{ij}$ is \emph{shaded} gray if $i+j$ is even, and white otherwise.
				
				We already used the coordinatisation of $[6a+2]^2$ as $(B_{i_1 j_1 },L_{i_2  j_2 })$. In these coordinates, the guide is our desired projection of an $(a, 1)$ knight's tour onto the space of levels and we are yet to prove that there is a corresponding path in each level. We keep coming back to Figure \ref{g4} because the $(4, 1)$ knight is enough to demonstrate the general case.
				
				Consider the subgraphs of our $(a, 1)$ knight's graph induced by a single level. Each subgraph is just a 2-dimensional grid graph whose vertices are blocks $B_{ij}$. The subtlety is that the sizes of this grid graph differ for different levels.
			
				\begin{itemize} \itemsep-1pt
					\item Levels $L_{00}$, $L_{10}$, $L_{01}$ and $L_{11}$ are $7 \times 7$ grid graphs (Fig. \ref{g7a} (a)). We call them \emph{odd} levels. These admit a Hamiltonian path between any two gray blocks. The proof can be found in Appendix A.
					\item Levels $L_{ij}$, where exactly one of the indices $i$ and $j$ is in $\{0, 1\}$ are $7 \times 6$ or $6 \times 7$ grid graphs (Fig \ref{g7a} (b) or (c)).
					\item The remaining levels are $6 \times 6$ grid graphs (Fig. \ref{g7a} (d)).					
				\end{itemize}
				We call the latter two \emph{even} levels, and they admit Hamiltonian paths between any two blocks of different shade.
				
			\begin{figure}[!htb]
				\centering
				\subfloat[labelformat=simple][]{
					\includegraphics[scale=.24]{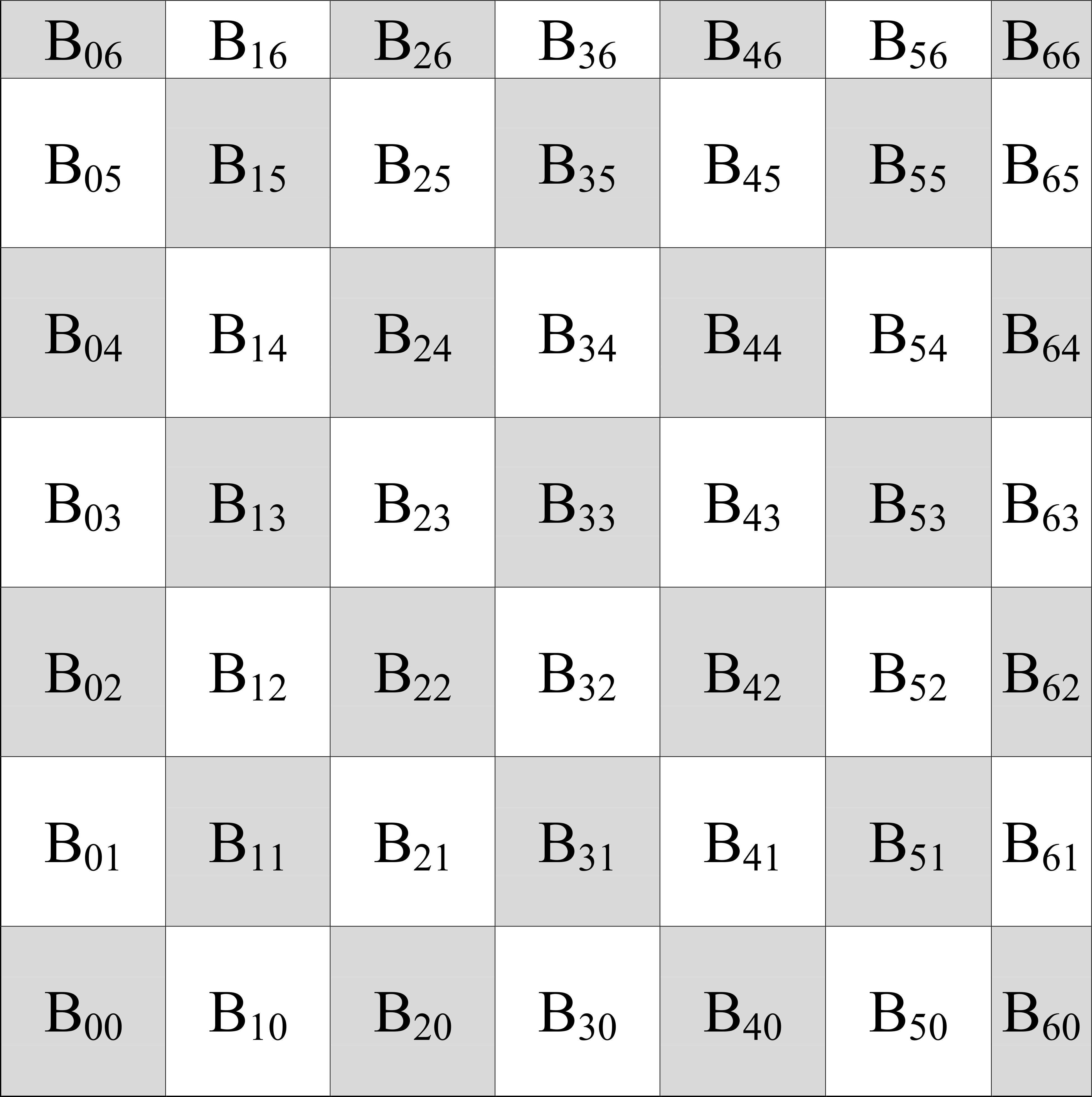}
					}
				\hspace{.5cm}
				\subfloat[labelformat=simple][]{
					\includegraphics[scale=.24]{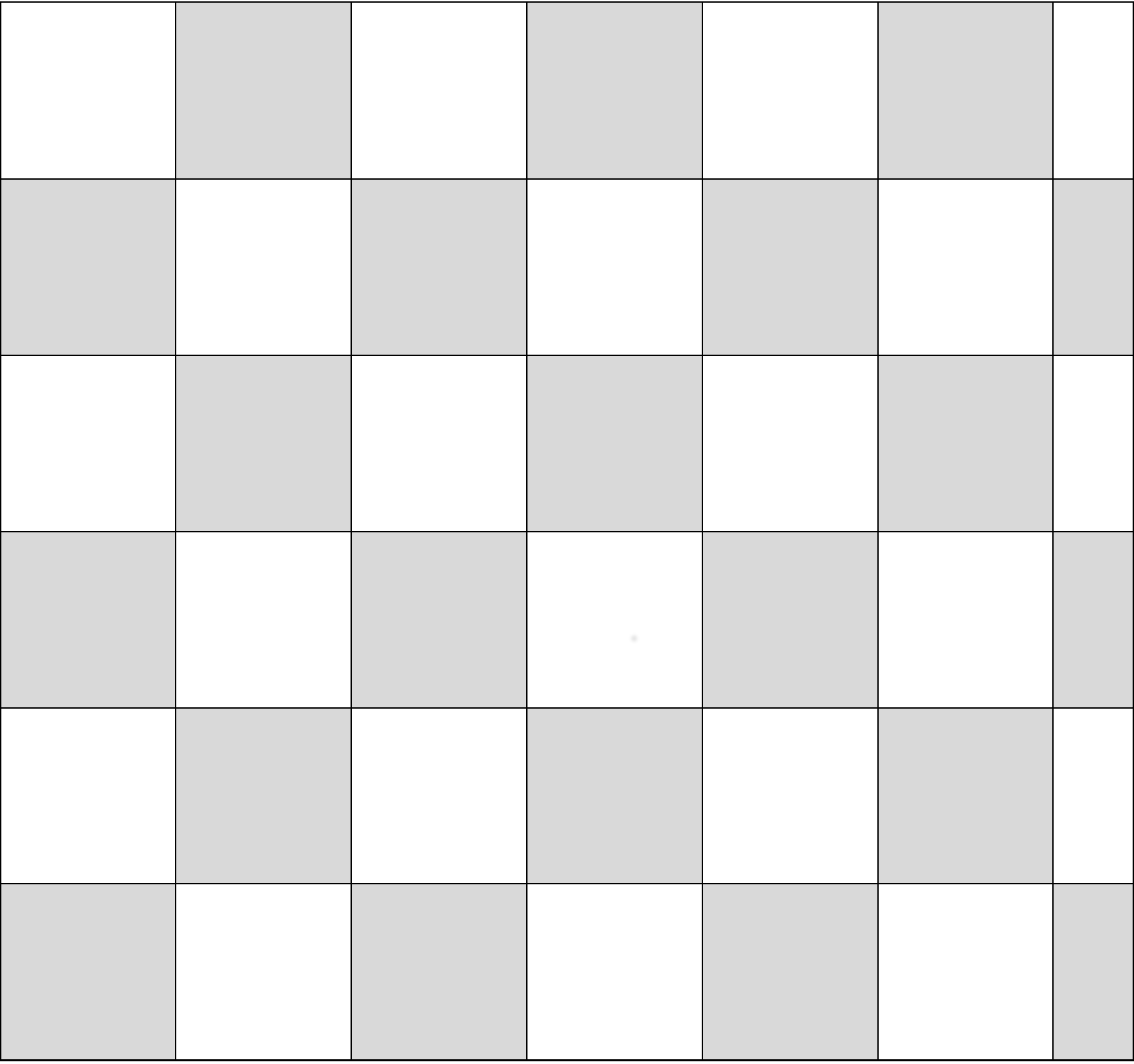}
					}
				\hspace{.5cm}
				\subfloat[labelformat=simple][]{
					\includegraphics[scale=.24]{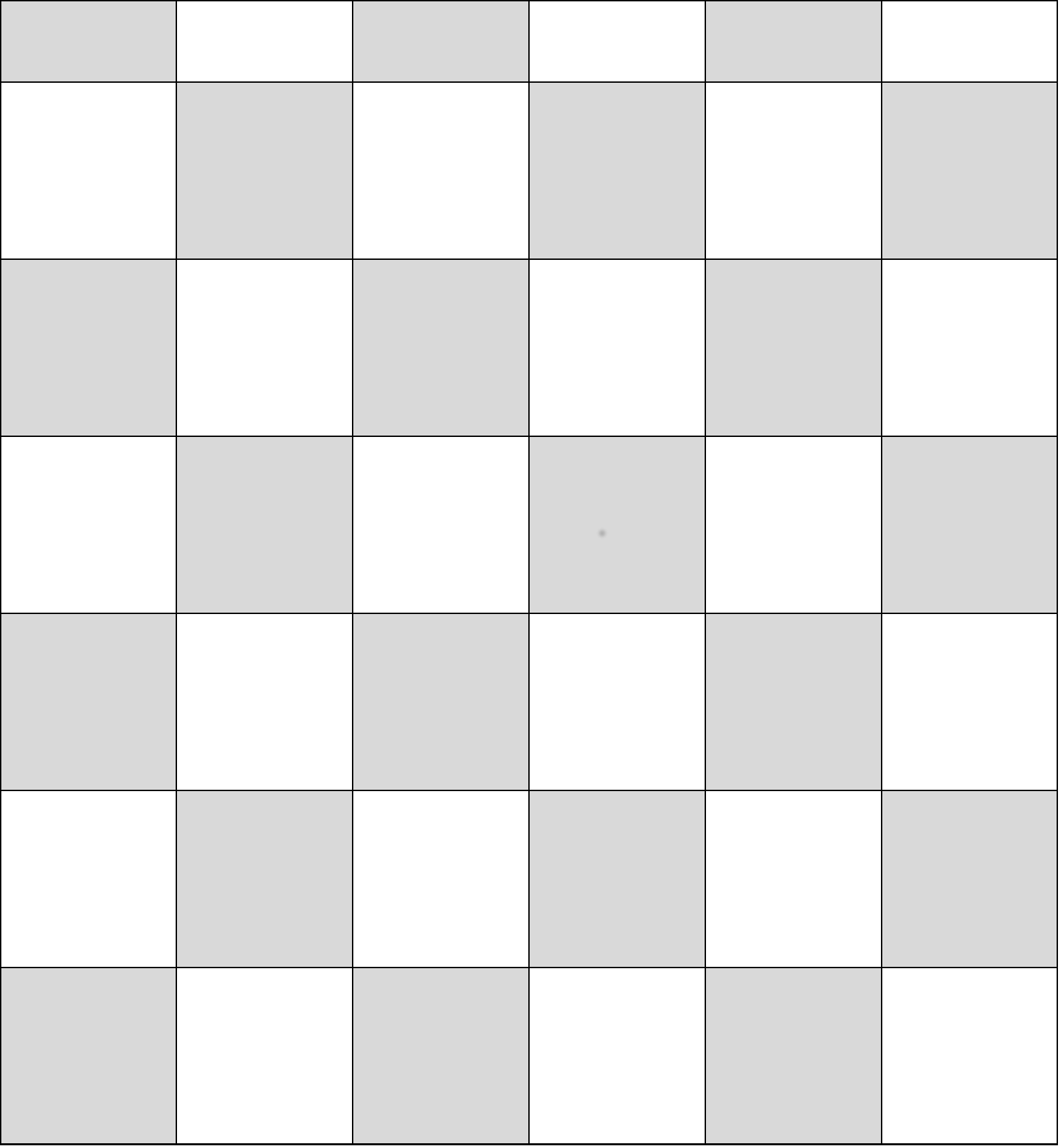}
					}		
				\hspace{.5cm}
				\subfloat[labelformat=simple][]{
					\includegraphics[scale=.24]{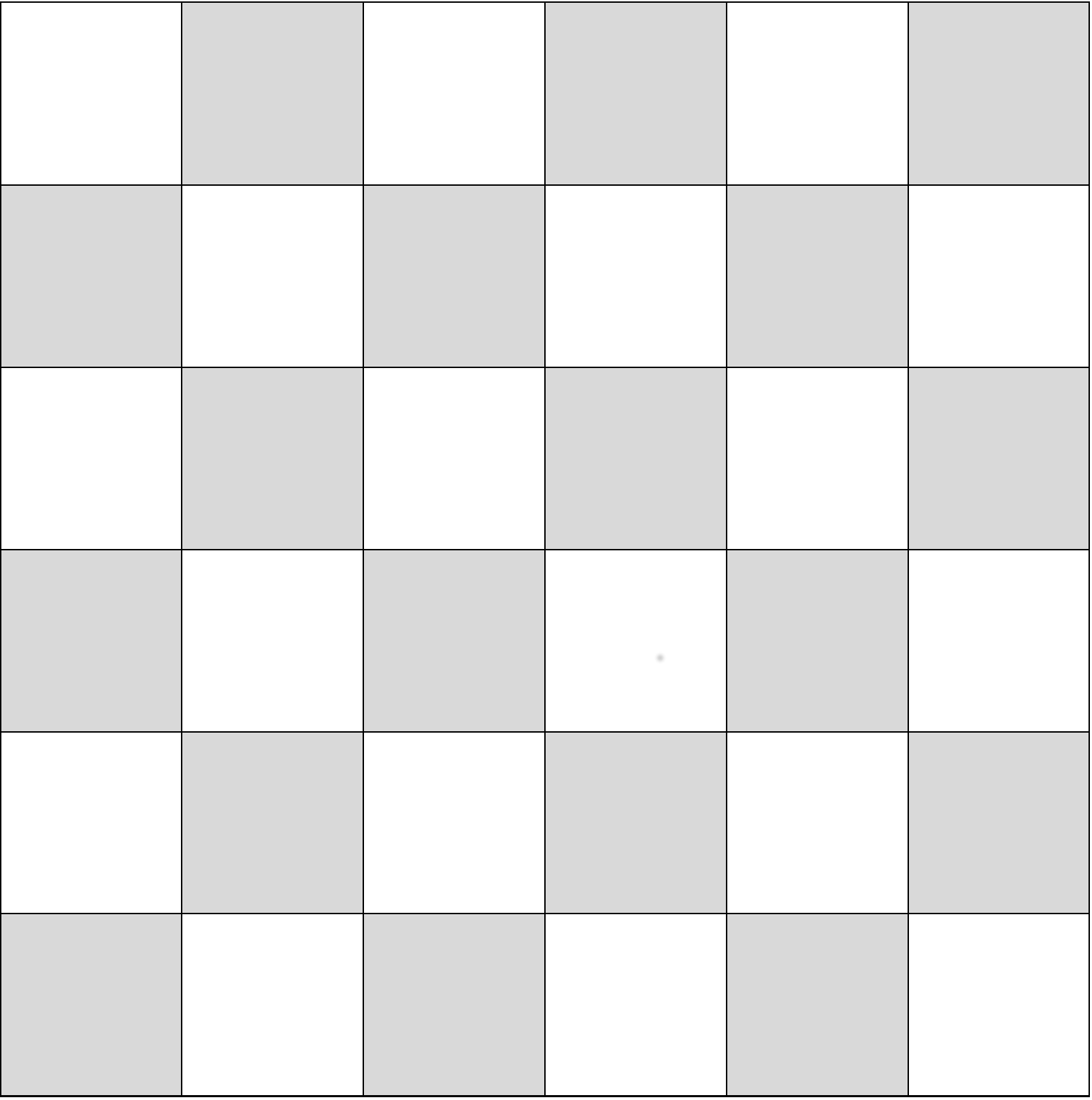}
					}		
					\caption{For any even $a$, within each level, our $[6a+2]^2$ chessboard is reduced to one of these grid graphs whose vertices are blocks. We \emph{shade} these blocks as a standard chessboard.\\
Recall, the difference in size arises because the incomplete blocks along the edge of $[6a+2]^2$ do not contain all levels (see Figure \ref{g4}).
} \label{g7a}					
			\end{figure}

					It remains to study adjacency between levels in more detail. Let the vertices $V$ and $V^\prime$ be adjacent in the $(a, 1)$ knight's graph and belong to distinct levels $L$ and $L^\prime$ respectively. Then the block containing $V$ is called a \emph{lift} from $L$ to $L^\prime$.
					
					The four horizontally lined blocks in Figure \ref{g7} are called \emph{transition region}, or simply \emph{T-region}, so T$= \{B_{ij}: i, j \in \{2, 3\} \}$. The vertically lined blocks are called \emph{T1-region}, so T1 $= \{B_{ij}: i, j \in \{1, 2, 3, 4\} \} \setminus $T.
				
				\begin{figure}[!ht]
					\centering
					\includegraphics[scale=.3]{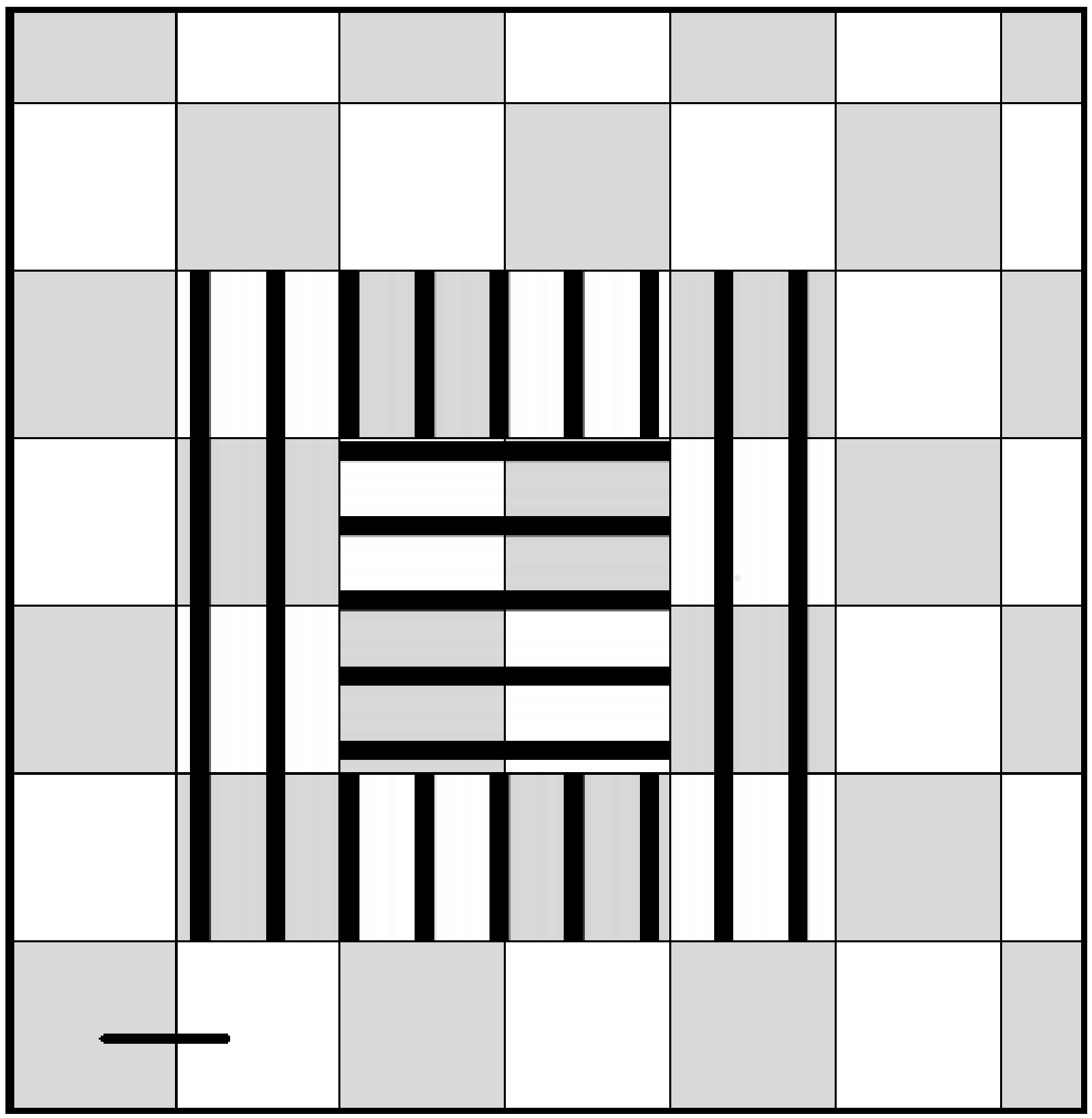}
					\caption{Blocks of a $[6a+2]^2$ chessboard, with T-region (inside) and T1-region (outside) indicated.}\label{g7}
				\end{figure}  				  

				\begin{lemma} \label{l4}
					Let $a$ be even. Given two adjacent levels $L$ and $L^\prime$, T-region contains lifts from $L$ to $L^\prime$ of either shade (white and gray).
				\end{lemma}
				\begin{proof}
					Let $L$ and $L^\prime$ be vertically related. Then the set of lifts from $L$ to $L^\prime$ is of form $\{B_{ij}:j \text{ even}\}$ or $\{B_{ij}:j \text{ odd}\}$ (from Figure \ref{g4}), say the prior. T-region contains two blocks in an even row, one of each shade. These are our required lifts from $L$ to $L^\prime$.
					
The set of lifts is equivalent for horizontally related levels, so the statement still holds.
				\end{proof}
				\begin{theorem} \label{t5}
					If $a$ is even, there exists a structured $(a, 1)$ knight's tour in $[6a+2]^2$.
				\end{theorem}
				\begin{proof}
					To construct this path, we will have to follow three graphs - Figure \ref{g7} is sufficient to illustrate moving within a single level (steps A1 - A3). The guide dictates the order in which levels are traversed. Finally, to see how the transitions  between levels (B1 - B2) reflect on the knight's position in Figure \ref{g7} (i.e. moving between blocks), we have to go back to the full chessboard (Figure \ref{g4} or \ref{g4n_72}).
					
					Our algorithm for implementing the guide will have the following steps.
					\begin{description} \itemsep-1pt
						\item[A1] 	Suppose we are at the start of an even level $E$, in one of the gray blocks within the T1-region. Let $E$ be followed by $L$ in the guide. This step comprises simply doing a Hamiltonian path on level $E$ and finishing in the white lift from $E$ to $L$ within the T-region. The existence of this lift is granted by Lemma \ref{l4}.\\
					Blocks $B_{00}$ and $B_{10}$ are adjacent in this Hamiltonian path because $B_{00}$ has only two neighbours ($B_{01}$ and $B_{10}$), and is not an endpoint.
						\item[A2]	Same as A1, except that we start in a white and finish in a gray block.
						\item[A3]	As above, except that we are traversing an odd level. The algorithm below guarantees that we start in a gray block within the T1-region, and we make sure to finish in a gray lift in T-region.\\
						Blocks $B_{00}$ and $B_{10}$ are also adjacent in this path, as indicated by the cord connecting them in Figure \ref{g7}.
					\end{description}	
						It is not necessary to distinguish between steps A1 - A3, but doing so makes it easier to track block shades along our knight's tour (see the algorithm below).
						
					We go back to Figure \ref{g4} and Lemma \ref{l2a} to notice that lifting from a level with relative position A to another level with position A reflects on Figure \ref{g7} as a single horizontal or vertical move. However, special lifting results in a diagonal change of block. This leads to the following steps.
					\begin{description}	\itemsep-1pt
						\item[B1]	We have just finished a level and are finding ourselves in a lift within T-region. The next move in the guide is not special, so we simply move horizontally or vertically into the T1-region while lifting.\\
					This lifting causes a switch from a gray block to white or vice versa.
						\item[B2]	We have just finished a level that is the last one with its relative unit-position (A, B, C or D). This means that we are in a gray lift within T-region, and the next move in the guide is special. This still leaves us with a choice - if we would like to lift vertically, e.g.~from $L_{00}$ to $L_{a-2, 1}$ we have a choice of direction vertically. Thus it is possible to move diagonally up or diagonally down. At least one of the resulting blocks will be in T1, so we choose that one.
					\end{description}
					
					Finally the algorithm is the following:
\renewcommand{\labelenumi}{{\arabic{enumi}.}}
\vspace{-5pt}			\begin{enumerate}	\itemsep-1pt
						\item The knight starts at a vertex from $L_{a-2, 1}$ which we also regard as the final one. This initial vertex belongs to $B_{33}$ - the gray block within T-region which is a lift from $L_{a-2, 1}$  to $L_{00}$.
						\item The knight makes the transition B2.
						\item A1 - B1 is repeated as long as the knight is on an even level (possibly zero times).
						\item The knight is on an odd level, so we perform step A3.
						\item B1 - A2 is repeated until all levels of the current position are covered (again, possibly zero times).
						\item If the knight is on a level with position A, B or C, go back to 2. Otherwise, go to 7.
						\item By choice of an appropriate path on level $L_{a-2, 1}$, we are back to the starting point.
					\end{enumerate}
\renewcommand{\labelenumi}{{\normalfont (\roman{enumi})}}
					To make sure the algorithm does construct a Hamiltonian cycle, we note that
					\vspace{-5 pt} \begin{itemize}	\itemsep-1pt
						\item The knight finishes each stage 1-6 inside a gray block, which guarantees that the next step is possible.
						\item Stages 2-6 are repeated four times - once for each relative unit-position in the guide.
						\item The key to validity of the algorithm is that each position (A, B, C, D) contains exactly one odd level and thus requires exactly one execution of step 4.
						\item It remains to show that the $(a, 1)$ knight's tour we constructed is structured. This is true by the fact that blocks $B_{00}$ and $B_{10}$ are adjacent when covering each level (steps A1-A3). Recall, the $a^2$ edges required in definition of a structured tour are exactly those lying between the two blocks in the bottom left corner.
					\end{itemize}
				\end{proof}				
				  	
					The third bullet gives the reason why $n \equiv \pm 2 \left( \text{mod } 2k \right)$ and potentially $n \equiv 0 \left( \text{mod }  2k \right)$ are the only cases in which we can hope to implement our guide - in other cases the number of odd levels is greater than four, requiring more than four special lifts.
					
					Even worse, for any larger $a$ there will be odd cycles from which there are no special moves (e.g. 22, 23, 32, 33 in case  $a=6$). Then our shading argument shows that for most board sizes there is no $(a, 1)$ knight's tour that covers each level entirely before moving on to the next one.
				
	\subsection{The main theorem}
		In analogy with link A, which is used for sequential concatenation of $(a, 1)$ knight's tours in Lemma 4, we define \emph{link B} to be the edge $\{(a, 0), (0, 1)\}$. Link B will be used to perform sequential concatenation in a slightly different context (Theorem \ref{t8}).
		\begin{corollary} \label{c7}
			For any even numbers $a$ and $n$, with $n \geq a(6a+2)$, there exists an $(a, 1)$ knight's tour in $[n]^2$. In addition, this tour can be chosen to contain link B.
		\end{corollary}
		\begin{proof}
			We induct on $n$ based on Lemma \ref{l4}, with the additional condition that our tours are structured.
			
			For the basis, we need structured $(a, 1)$ knight's tours for each value of $n$ modulo $2a$. Theorem \ref{t5} and Lemma \ref{l2} give us the existence of a structured tour for values $n=6a+2,\ 2(6a+2),\ 3(6a+2),\dots a(6a+2)$. Now our inductive steps of length $2a$ (Lemma \ref{l1}) starting from these basic cases cover all sufficiently large natural numbers.
			
			As for link B, it is the edge $\{ (B_{00}, L_{01}), (B_{10}, L_{01})\}$, so it exists in any structured tour. This completes the proof.
		\end{proof}
		As noted before, extending into $d$ dimensions is straightforward in comparison with the work done so far. The only additional structure we require from a 2-dimensional knight's tour is link B.
		\begin{theorem} \label{t8}
			For any even values of $a$ and $n$, with $n \geq a(6a+2)$, there exists an $(a, 1)$ knight's tour in the grid $[n]^d$.
		\end{theorem}
		\begin{proof}
			The construction uses sequential concatenation and is thus quite similar to the proof of Lemma \ref{l2}.
			
			Number the points $[n]^d$ by $(x_1,\ x_2, (x_3, \dots x_d))$ to emphasise the idea that the grid is divided into subgraphs called \emph{floors} and labelled by $\mathbf{p}=(x_3, \dots x_d)$, each floor being a copy of $[n]^2$ (i.e.~canonically isomorphic to it).
			
			Construct any Hamiltonian path $\mathbf{p}_1,\ \mathbf{p}_2,\  \dots \mathbf{p}_l$ of the $(d-2)$-dimensional grid graph (i.e.~such that each two neighbouring vertices differ by just a unit vector along any axis, which is denoted by $\mathbf{e}_i$). This will be used as our sequence for sequential concatenation.
			
			On each level $\mathbf{p}_i$, use Corollary \ref{c7} to display an $(a, 1)$ knight's tour containing link B. Consider levels $\mathbf{p}_i$ and $\mathbf{p}_{i+1}$ for any  i. The following edges make up a bridge between them:\\
$\{(a, 0, \mathbf{p}_i), (0, 1, \mathbf{p}_i) \}$  (link B) and\\
$\{(0, 0, \mathbf{p}_{i+1}), (a, 1, \mathbf{p}_{i+1}) \}$ (must exist because the corner of floor $\mathbf{p}_{i+1}$ has degree 2).

For each $i$ starting from 1, use the bridge we just constructed to concatenate $(a, 1)$ knight's tours on levels $\mathbf{p}_i$ and $\mathbf{p}_{i+1}$. The result is an $(a, 1)$ knight's tour in the $d$-dimensional grid, as required.
		\end{proof}
		  We remark that this proof is valid if we only assume that one 2-face of the grid is square and greater than $a(6a+2)$. It can be modified into induction on $d$, which is equivalent to choosing a particular path $\mathbf{p}_1,\ \mathbf{p}_2,\dots \mathbf{p}_l$.
			
			This proof makes concatenation the universal concept in constructing Hamiltonian cycles, used for
				\vspace{-0.5\baselineskip}
				\begin{itemize}\addtolength{\itemsep}{-0.5\baselineskip}
					\item connecting the rim of the board to the middle (Lemma \ref{l1}),
					\item replicating $(a,\ 1)$ knight's tours (Lemma \ref{l2}),
					\item connecting $(a,\ 1)$ knight's tours on different flours for higher-dimensional knight's graphs (Theorem \ref{t8}).
				\end{itemize}	
		
\section{The \texorpdfstring{$(a, b)$}{(a, b)} knight}
	We first generalise the proof of Theorem \ref{t8}, which enables us to extend 2-dimensional $(a, b)$ knight's tours into $d$ dimensions. Lemma \ref{l9} iterates the argument we used to prove Theorem \ref{t8} twice, so it is still convenient to prove the two results separately.
			
	\subsection{Extending the \texorpdfstring{$(a, b)$}{(a, b)} knight's tour into \texorpdfstring{$d$}{d} dimensions}
			Under the assumption that an $(a, b)$ knight's tour exists in two dimensions, we divide $[n]^d$ into 2-dimensional floors as in the proof of Theorem \ref{t8}. Our aim is to connect all the levels, by which we mean concatenating the corresponding knight's tours. 
			
			Link B needs to be replaced by two links: link $\alpha$  is given by $\{(0, b), (a, 0)\}$, and link $\beta$ by $\{(n-1, a), (n-b-1, 0)\}$. Link $\alpha$ is used to connect floors $\mathbf{p}$ and $\mathbf{p}^\prime$ displaced by $a\mathbf{e}_i$. The following edges form a bridge between them:\\
			 $\{(0, 0, \mathbf{p}), (a, b, \mathbf{p})\}$ and\\
			 $\{(0, b, \mathbf{p}^\prime), (a, 0, \mathbf{p}^\prime)\}$.\\
			Note that link $\alpha$ is used up only at level $\mathbf{p}^\prime$, which allows for sequential concatenation that was already used in Lemma \ref{l2} and Theorem \ref{t8}.
			
			Equivalently, we use link $\beta$ to connect levels that differ by $b\mathbf{e}_i$. A 2-dimensional $(a, b)$ knight's tour containing both links is called a \emph{linked} knight's tour.
			
			Since $a$ and $b$ are assumed to be coprime, between any two floors there is a path which consists of steps of length $a$ and $b$ along the principal axes. Lemma \ref{l9} gives a way of connecting the floors in an orderly way using the idea of sequential concatenation.
			\begin{lemma} \label{l9}
				Let there exist a linked $(a, b)$ knight's tour in $[n]^2$, with $n \geq a+b$. Then for any $d$, the grid $[n]^d$ admits a knight's tour.
			\end{lemma}
			\begin{proof}
				As before, display an $(a, b)$ knight's tour on each 2-dimensional floor.
				
				We extend the concept of congruence modulo $a$ to $\varmathbb{Z}^{d-2}$ and denote the congruence class of $\mathbf{p}$ by $\bar{\mathbf{p}}$. Furthermore, let $r(\mathbf{p})$ be the residue of $\mathbf{p}$ modulo $a$, defined by 
				$$ r(\mathbf{p})  \in   \bar{ \mathbf{p} }  \cap \{0,\ 1, \dots a-1\}^{d-2}. $$
				For any floor $\mathbf{p}$, the graph with vertices from $\left[ n \right] ^{d-2} \cap \bar{ \mathbf{p} } $ and edges of form $a\mathbf{e}_i$ is a grid graph, so it admits a Hamiltonian path $P$. Using link $\alpha$, we can sequentially concatenate the $(a, b)$ knight's tours along $P$, exactly as in Theorem \ref{t8}.
				
				The result is a set of vertex-disjoint cycles $\{T_{ \bar{ \mathbf{p} }}  \mid \mathbf{p} \in \{0,\ 1, \dots a-1\}^{d-2} \}$, where each cycle is a tour in the corresponding congruence class. To concatenate these cycles, we follow yet another Hamiltonian path in a grid graph. This time the vertices are $\{0, b,  \dots b(a-1)\}^{d-2}$ and the edges are of form $b\mathbf{e}_i$. We represent the Hamiltonian path by a sequence of vertices $\mathbf{p}_1,\ \mathbf{p}_2,\dots \mathbf{p}_l$.
				
				We use the bridge between levels $r(\mathbf{p}_i)$   and $r(\mathbf{p}_i)+(\mathbf{p}_{i+1} - \mathbf{p}_i)$  to concatenate cycles $T_{ \bar{ \mathbf{p}}_{i+1} }$ and $T_{ \bar{ \mathbf{p}}_{i} }$. To show that this is possible and the result is an $(a, b)$ knight's tour, recall that
				\begin{itemize}	\itemsep-1pt
					\item $a$ and $b$ are coprime since the 2-dimensional $(a, b)$ knight's graph admits a Hamiltonan cycle. It follows that each congruence class modulo $a$ contains exactly one member of $\{0, b,  \dots b(a-1)\}^{d-2}$, so $\bar{ \mathbf{p}}_1 ,\ \bar{ \mathbf{p}}_2 ,\dots \bar{ \mathbf{p}}_l $ is indeed a Hamiltonian path of congruence classes.
					\item Floors $r(\mathbf{p}_i )$    and  $r(\mathbf{p}_i)+(\mathbf{p}_{i+1} - \mathbf{p}_i)$ differ by $b\mathbf{e}_i$, so they can be connected using link $\beta$. Floor $r(\mathbf{p}_i)+(\mathbf{p}_{i+1} - \mathbf{p}_i)$ exists by the condition $n \geq a+b$.
					\item The fact that we previously used up link $\alpha$ does not interfere with the existence of link $\beta$ because they are defined to lie in different corners.
				\end{itemize}
			\end{proof}

				Link $\alpha$ and link $\beta$ can be modified if we find a particular 2-dimensional $(a, b)$ knight's tour that does not contain them. The bridge\\
			 $\{(0, 0, \mathbf{p}), (a, b, \mathbf{p})\}$ and\\
			 $\{(0, b, \mathbf{p}^\prime), (a, 0, \mathbf{p}^\prime)\}$\\
				and its symmetric images are particularly convenient because they work for both $\mathbf{p}-\mathbf{p}^\prime =  a\mathbf{e}_i$ and (with reordering the vertices) $\mathbf{p}-\mathbf{p}^\prime = b\mathbf{e}_i$.
		
		\subsection{\texorpdfstring{The $(2, 3)$}{(2, 3)} knight}
			For an $(a, b)$ knight with $a$ and $b$ both greater than 1, we do not have a partition into blocks and levels analogous to Figure \ref{g4}. This means that for now, basic cases are constructed using heuristic computer programs and do not possess any structure except that dictated by the corners.
			
			We do use analogues of Lemma \ref{l1} and Lemma \ref{l2} for the inductive step, but have to work harder to obtain a suitable colouring of the rim.
			\paragraph{Basic case}
				The figure below shows a $(2, 3)$ knight's tour for $n=10$, generated by a computer program. It is easy to show (using restrictions on vertices near the corner) that no $(2, 3)$ knight's tour exists in smaller chessboards.
				\begin{center}
					\includegraphics[scale=0.7]{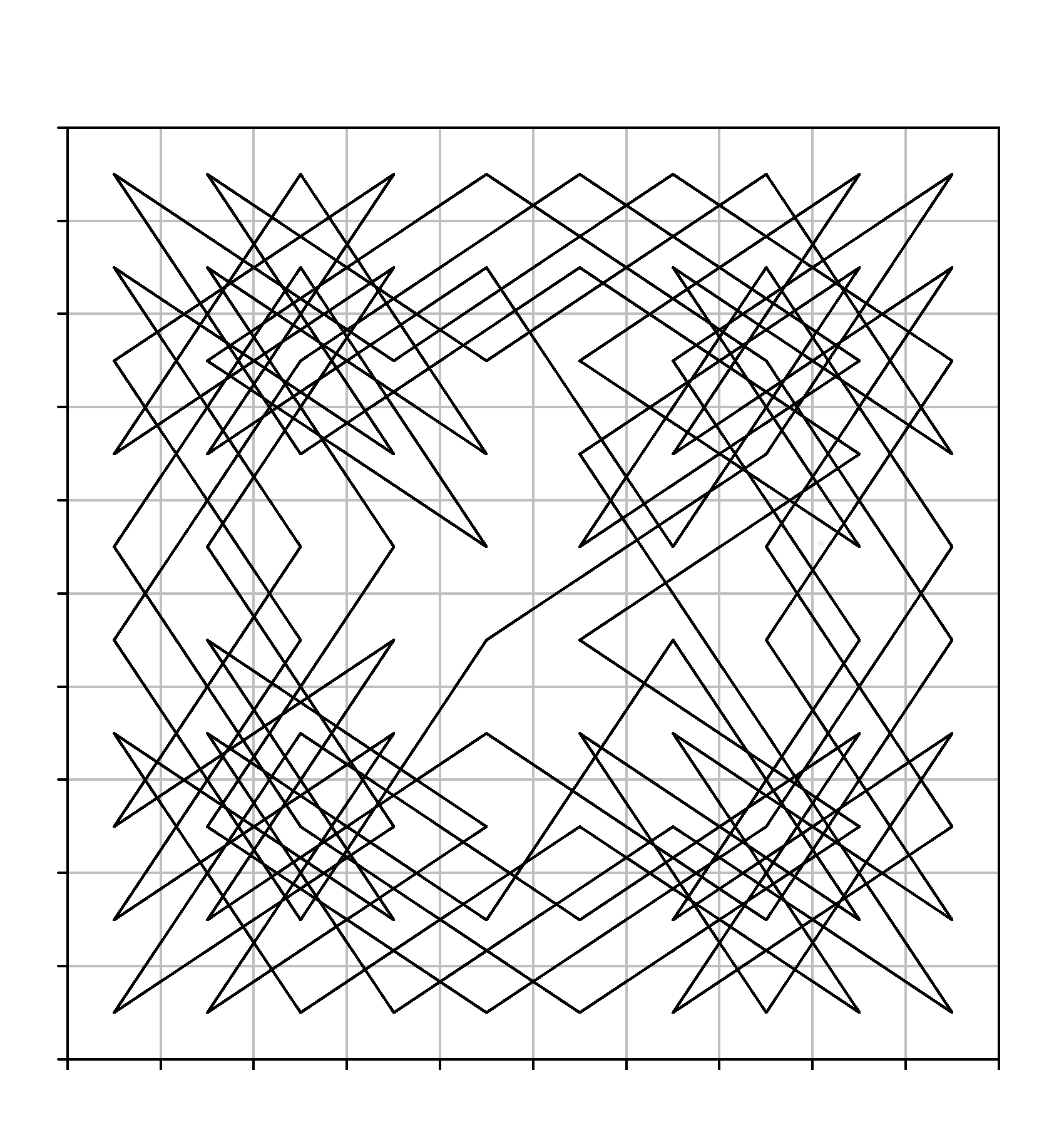}
				\end{center}
			\paragraph{Rim of width 6}
				
				There is a natural colouring of a straight band of width 6 into 12 colours, with each colour following a knight's path. It is based, as before, on translating blocks, but they are now rectangular\footnote{Recall, the side length of the original board being even is a necessary condition for the existence of an $(a, b)$ knight's tour.} (as opposed to squares from the previous section).
			\begin{figure}[!hp]
				\centering
				\includegraphics[scale=.8]{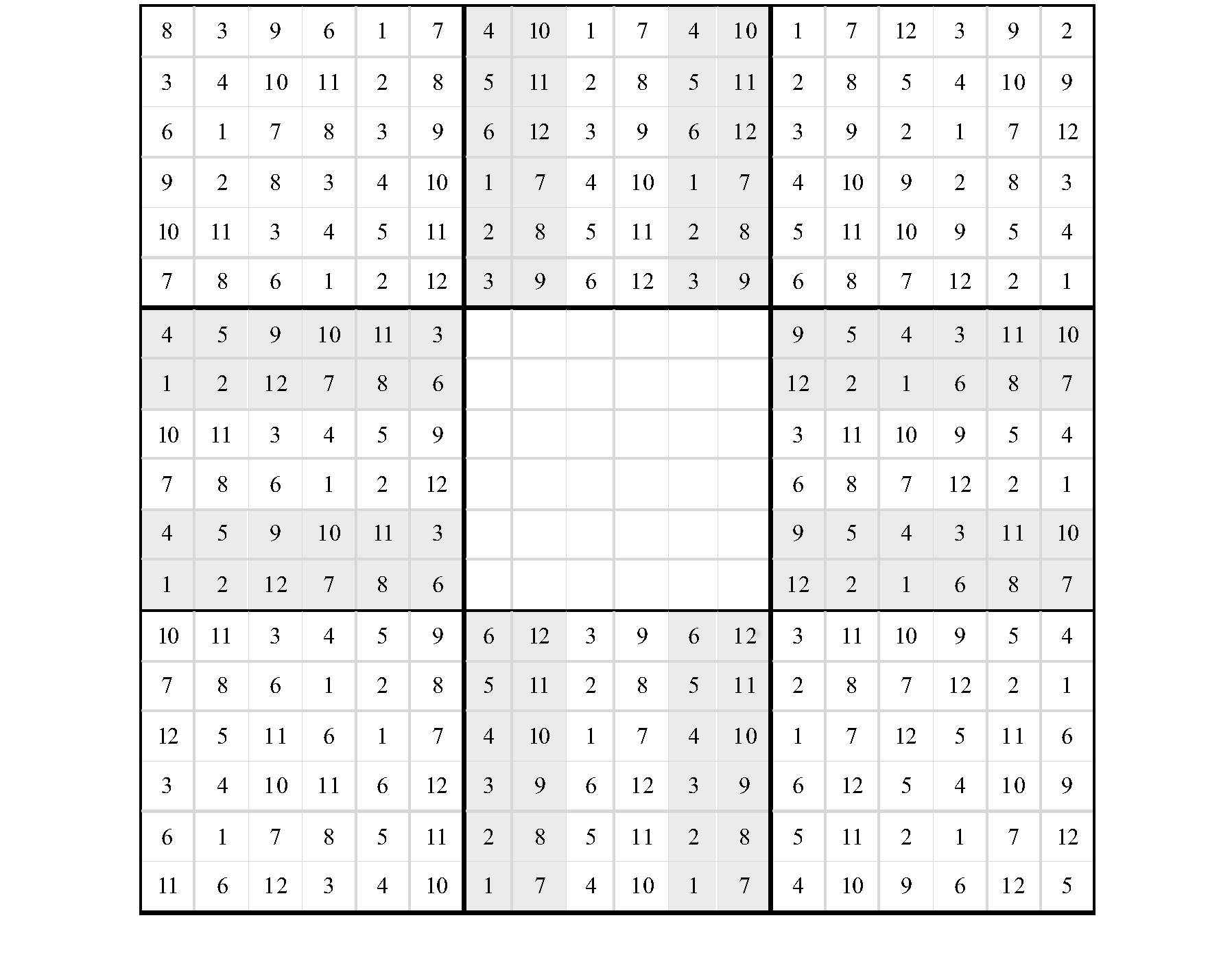}
				\caption{Cycles in the rim of $[n]^2$, for $n \equiv 2 \left( \text{mod }  4 \right)$. The translating blocks are now $2 \times 6$ rectangles. We refer to sections of the board as \emph{corners} (squares of side length 6), \emph{bands} (straight sections of the rim tiled by blocks) and the \emph{middle}.\\
As for the $(a, 1)$ knight, colours are numbered based on the bottom left block.
} \label{g11}
			\end{figure}
				
				This implies that the colouring has a pattern as long as we move horizontally between the blocks, but the path of each colour around a corner is not initially imposed. The corners below are simply a colouring of the square of side length 6 made up so that all 12 colours can `emerge' on either side of the square. We now show that this gives a partition of the rim into 12 cyclic graphs.
			
			\begin{figure}[!ht]
				\centering
				\includegraphics[scale=.8]{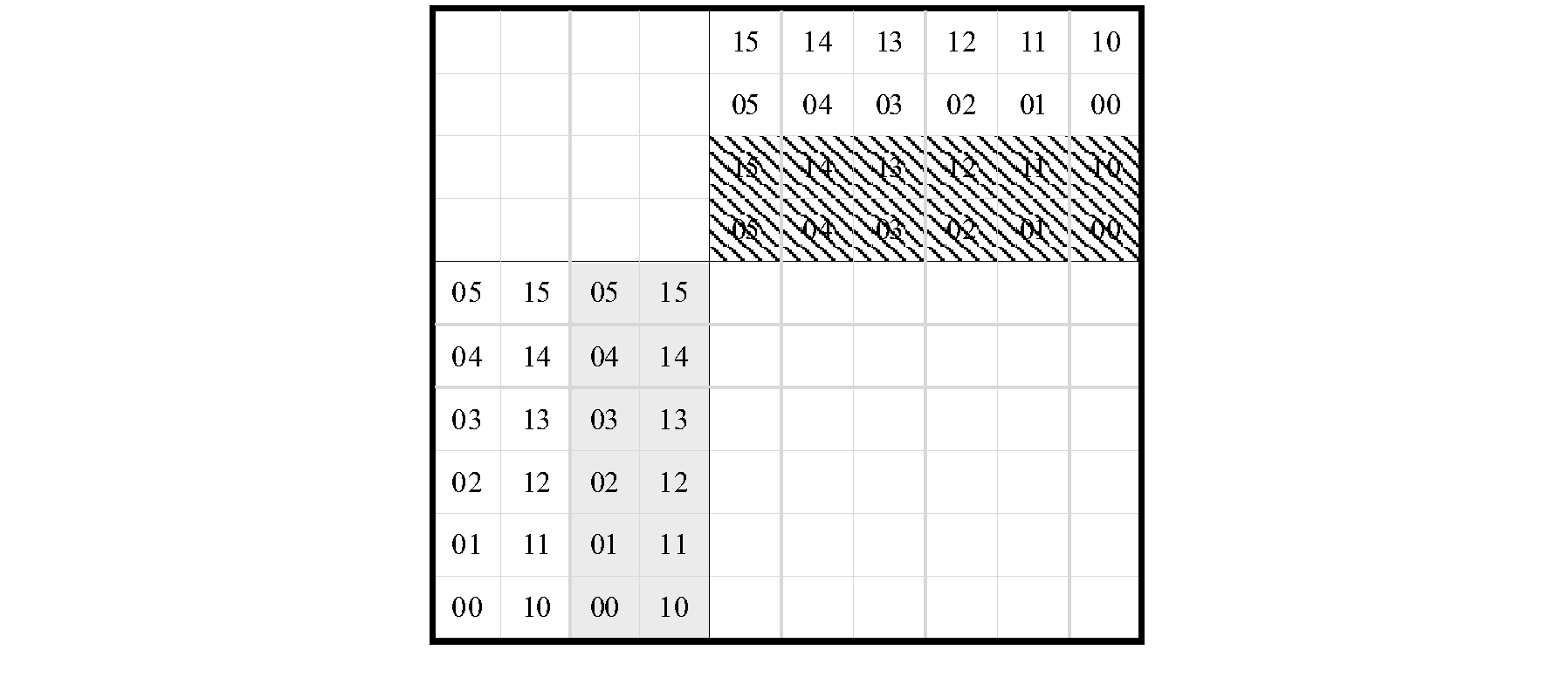}
				\caption{Block positions illustrated in four blocks near the bottom right corner. The block-position of each colour in the gray block determines how it behaves around the corner. Two highlighted blocks induce a permutation of colours.
} \label{g12}
			\end{figure}

					Annotate each rectangular block internally as in Figure \ref{g12}. Given a certain block, we assign \emph{block-positions}\footnote{Block-position is a concept similar to unit-position defined in Section 2.2} to each colour depending on which square the colour occupies. Define the permutation $f: \{1,\ 2, \dots 12\} \longmapsto  \{1,\ 2, \dots 12\}$ by $f(i)=j$ if colour $j$ in the lined block (i.e. after the corner) has the same block-position as colour $i$ in the gray block (before the corner).

			For $n \equiv 2 \left( \text{mod }  4 \right)$ necessary and sufficient condition for each path to come back to its original block-position is that $f^4=id$. This is because the number of blocks in each band is odd (e.g.~3 blocks in each band in Figure \ref{g11}), and so block-position of each colour is unchanged after traversing the band. For the bottom right corner shown in Figure \ref{g12}, $f = (1\, 4\, 10\, 7)  (2\, 5\, 11\, 8)(3\, 9\, 12\, 6)$ , so the condition $f^4 = id$ is satisfied.
			
			If the number of blocks in each band is even ($n \equiv 0 \left( \text{mod }  4 \right)$), the colours swap block-positions in pairs as they propagate through the band. So our consistency condition is altered to \\ $\left[ (1\, 4)(2\, 5)(3\, 6)(7\, 10)(8\, 11)(9\, 12)f \right]^4 = id$, which is also true. Denote the cycle formed by colour $i$ by $C_i$.
			
			We are now ready to prove the equivalent of Lemma \ref{l2}. Notice that we now abandon the strategy of connecting each cycle to the middle separately. The trick in Lemma \ref{l10} is to keep creating the edges we need for the next bridge. It is worth noting that this version requires no additional structure for the initial (2, 3) knight's tour.
			\begin{lemma} \label{l10}
				If $[n]^2$ admits a $(2, 3)$ knight's tour, then so does one of side length $[n+12]^2$.
			\end{lemma}
			\begin{proof}
				As before, given the board $[n+12]^2$, construct the rim of width $6$ and use the assumption to form a $(2, 3)$ knight's tour in the middle.
				
	We first concatenate the cycles $C_1$, $C_4$, $C_9$ and $C_{12}$ using the bridges near the bottom left corner as shown in Figure \ref{g13}. Call the resulting cycle \emph{cluster 1}. Do the same with cycles $C_3$, $C_6$, $C_7$ and $C_{10}$ to form \emph{cluster 2}.
				
			\begin{figure}[!ht]
				\centering
				\includegraphics[scale=.8]{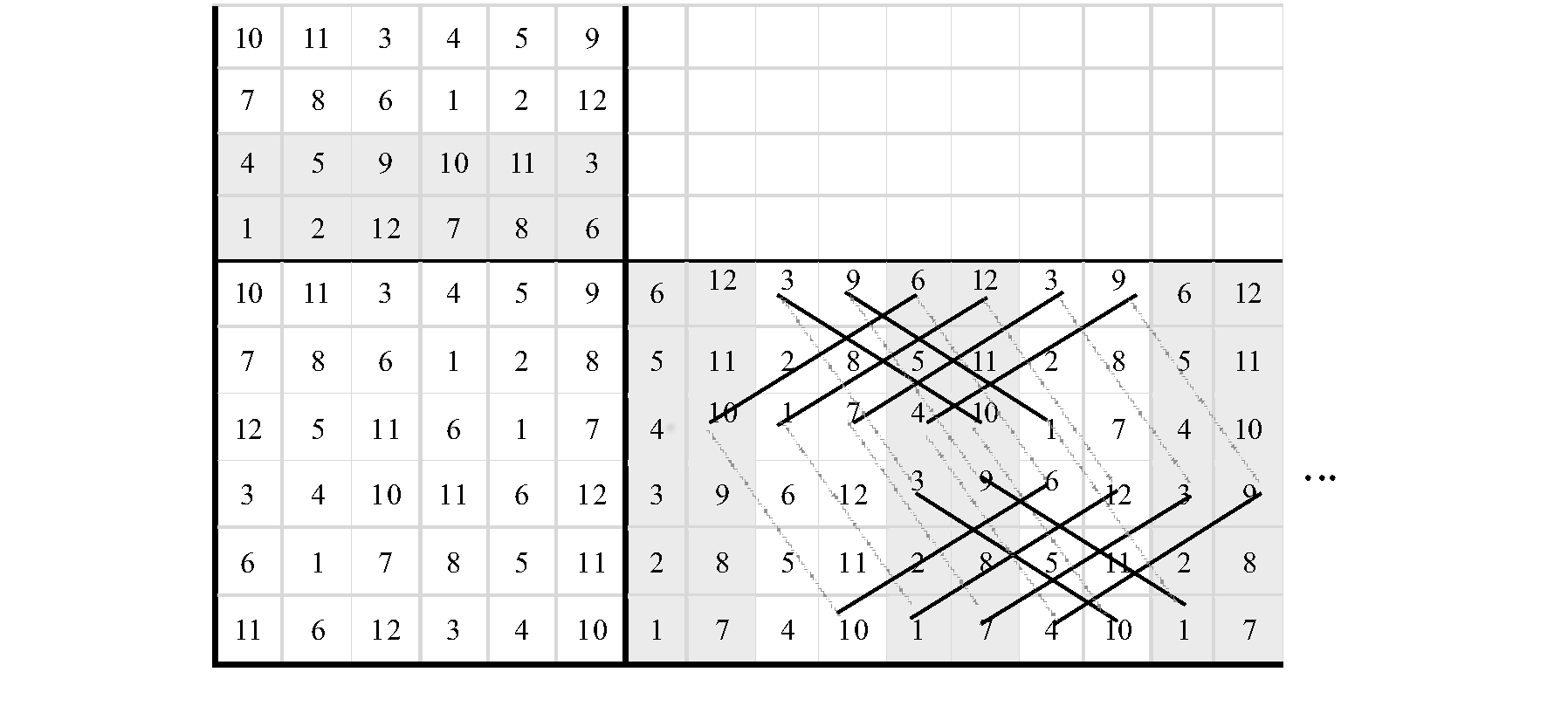}
				\caption{Concatenating cycles to construct clusters 1 and 2 - the dotted edges form bridges which are replaced by black edges in the process.\\
Regardless of the size of the board, we always use the first five blocks (in positive direction) for the concatenation. The figure only shows the relevant section of the chessboard.
} \label{g13}
			\end{figure}
			\begin{figure}[!htb]
				\centering
				\includegraphics[scale=.8]{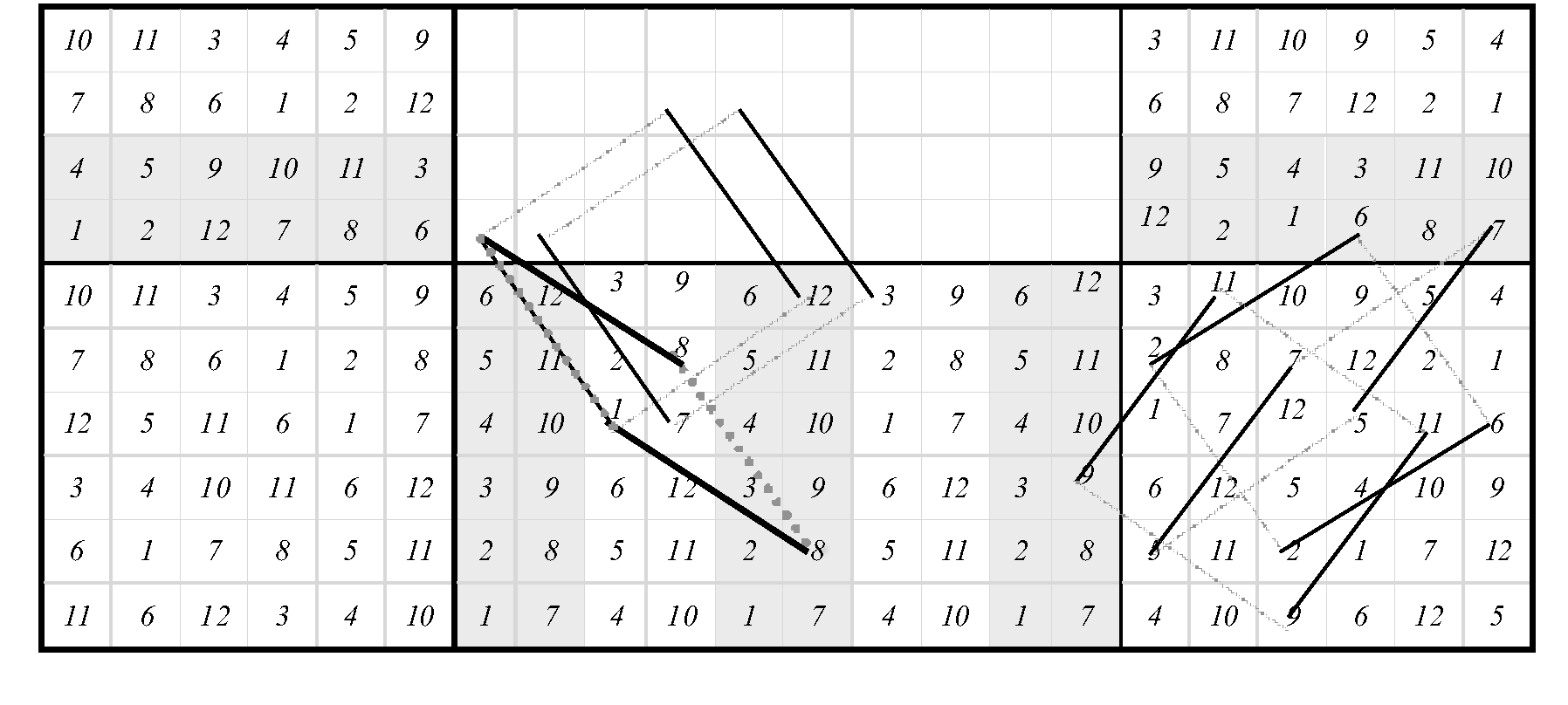}
				\caption{Appending cycles $C_2$, $C_5$ and $C_{11}$ to clusters in the bottom right corner.\\
	    Concatenating clusters with the middle in bottom left corner. The edge $\{(9, 4), (7, 7)\}$ is then used up to connect $C_8$ to the rest of the board (thick edges).
} \label{g14}
			\end{figure}
			The equivalent cannot be done with 2, 5, 8 and 11, but we can use the corner structure to add $C_2$ to cluster 2 (Fig. \ref{g14}). We do the same with cycles $C_5$ and $C_{11}$ (the reason why the solution is structured in this way is given below). We have to emphasise that this concatenation happens in the bottom right corner regardless of the size of the board.

			Now use the following bridge to connect cluster 1 to the middle:\\
			$\{(8, 3), (11, 5)\}$ (exists from concatenation of $C_1$ and $C_{12}$), and\\
			$\{(6, 6), (9, 8)\}$ (must exist in the $(2, 3)$ knight's tour in the middle).
			
			For cluster 2, the bridge is formed by the same mechanism: \\
			$\{(9, 3), (12, 5)\}$ and\\
			$\{(7, 6), (10, 8)\}.$
			
		We now have two vertex-disjoint cycles covering the chessboard - a small one given just by colour 8, and a large one covering the rest. One of the edges we just created, $\{(8, 3), (6, 6)\}$ can be used along with $\{(11, 1), (9, 4)\}$ (dotted gray lines) as a bridge between two cycles, giving us the final solution.
			\end{proof}
		
		\begin{corollary}
				For sufficiently large even $n$ and $d\geq 2$, there is a $(2, 3)$ knight's tour in $[n]^d$.
		\end{corollary}
			\begin{proof}
				Refer to Figure \ref{g12}. Link $\alpha$ and link $\beta$ exist as part of the turn (corner square of side length 6), so we can apply Lemma \ref{l9} to our (2, 3) knight's tour in $[n]^2$.
			\end{proof}
			
\newpage

		\subsection{\texorpdfstring{The $(2, 5)$}{(2, 5)} knight}
			We use a completely different method for extending (2, 5) knight's tours. It produces more `localised' (2, 5) knight's tours, similar to those found in references, but we still rely on sequential concatenation as the key principle.
			\subsubsection{Assembling rectangular blocks}
				This solution is included in the paper because it uses a new idea of assembling rectangular boards, $[m]\times [n]$. Lemma \ref{l13} reduces the $(a, b)$ knight's tour problem to finding certain `small' knight's tours.

				We consider an $(a, b)$ knight's graph in $[m]\times [n]$. Let $a>b$. We call the edge $\{(a-1, b), (a-b-1, a+b)\}$ and all its symmetric images \emph{link H}. Similarly, $\{(b, a-1), (a+b, a-b-1 )\}$ is called \emph{link V}. For $m = n$, these two edges coincide under reflection in the main diagonal.
				\begin{lemma} \label{l12}
				  Assume that there are $(a, b)$ knight's tours of  $[m_1] \times [n]$ and $ [m_2] \times [n]$ containing link H and link V. Let $\gcd(m_1,   m_2 )=2$. For a sufficiently large even number $k$, the $[k] \times [n]$ admits an $(a, b)$ knight's tour. In addition, this tour can be chosen to contain link V.
				\end{lemma}
				\begin{proof}
				  Any sufficiently large $k$ can be written as a linear combination of $m_1$ and $m_2$, which yields a partition of the board into a sequence of blocks of length $m_1$ and $m_2$. We construct $(a, b)$ knight's tours in these blocks and sequentially concatenate them using link H, as shown in the diagram.
				  
					\centering \includegraphics{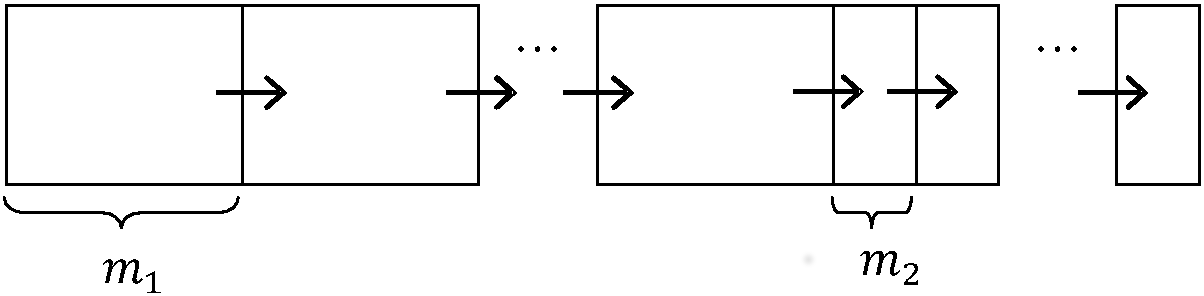}
				  
				  There is no assumption on the tour in the leftmost block, so the new $(a, b)$ knight's tour can inherit both link H and link V from it.
				\end{proof}
				\begin{lemma} \label{l13}
					Assume that there are $(a, b)$ knight's tours of  $[m] \times [n]$, $[n] \times [n]$ and $[m] \times [m]$, each containing link H and link V. Let $\gcd(m, n)=2$. For sufficiently large even numbers $k$ and $l$, $[k] \times [l]$ board admits an $(a, b)$ knight's tour. In addition, this knight's tour can be chosen to contain link V.
				\end{lemma}
				\begin{proof}
					First construct knight's tours in $[k] \times [n]$ and $[k] \times [m]$ as described in Lemma \ref{l12}. Then we can use the same argument to construct the required knight's tour in a $[k] \times [l]$ board. 
				\end{proof}
				
				\subsubsection{Square of side length 20}
					The first useful result for the $(2, 5)$ knight is the partition of the structure found on $[20] \times [10]$ into 16 vertex-disjoint cyclic graphs. Such a board will now be called a \emph{brick}.
					
					\begin{center}
						\includegraphics[scale=0.9]{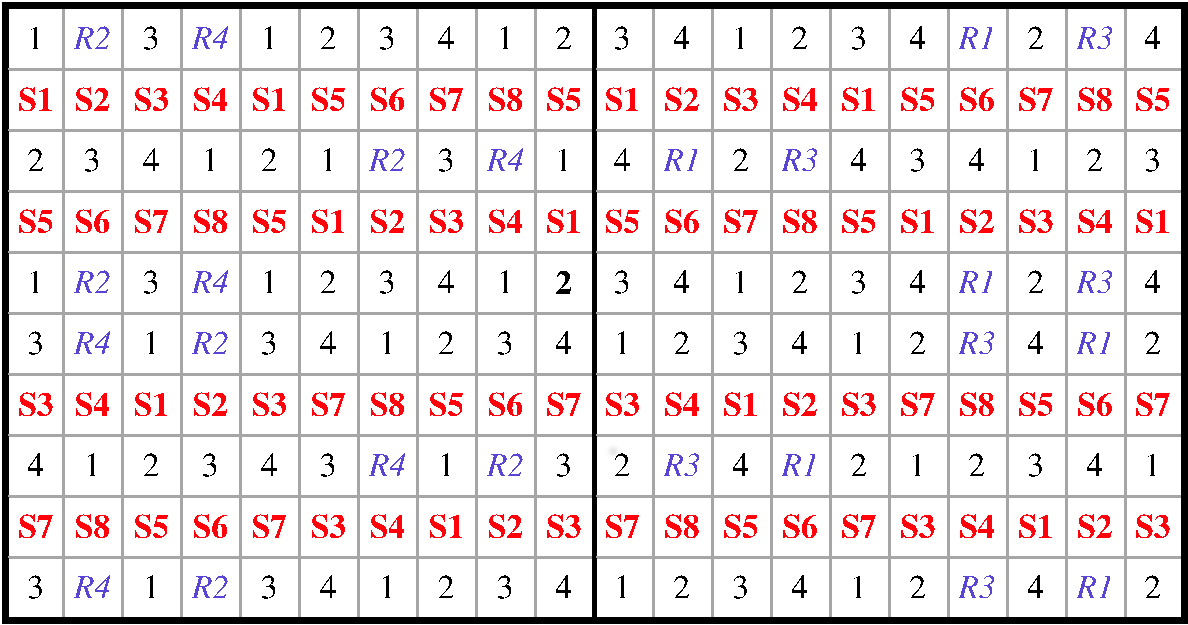}
					\end{center}
					
					While concatenating these cycles, we will have to keep track of which edges they contain, so we illustrate each cycle individually.
					\begin{center}
						\includegraphics[scale=0.95]{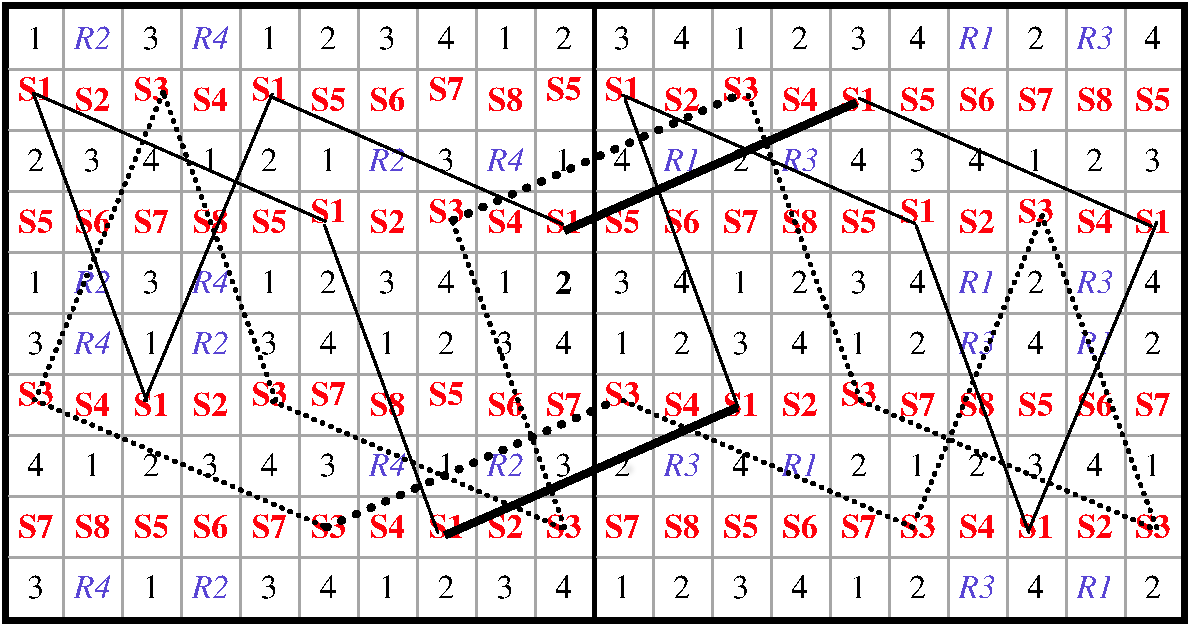}\\
						\small{Cycles $S1$ and $S3$. $S5$ and $S7$ are their mirror images.}
						\vspace{10pt}
						
						\includegraphics[scale=0.95]{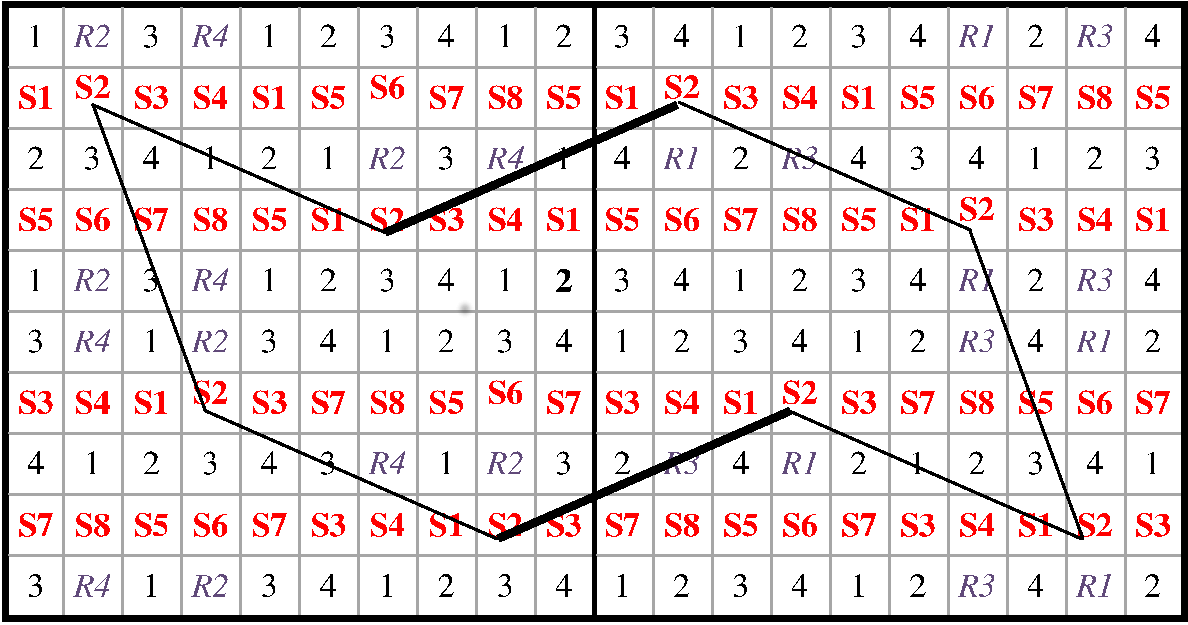}\\
						\small{Cycles $S$2, $S4$, $S6$ and $S8$.}
						
						\includegraphics[scale=0.9]{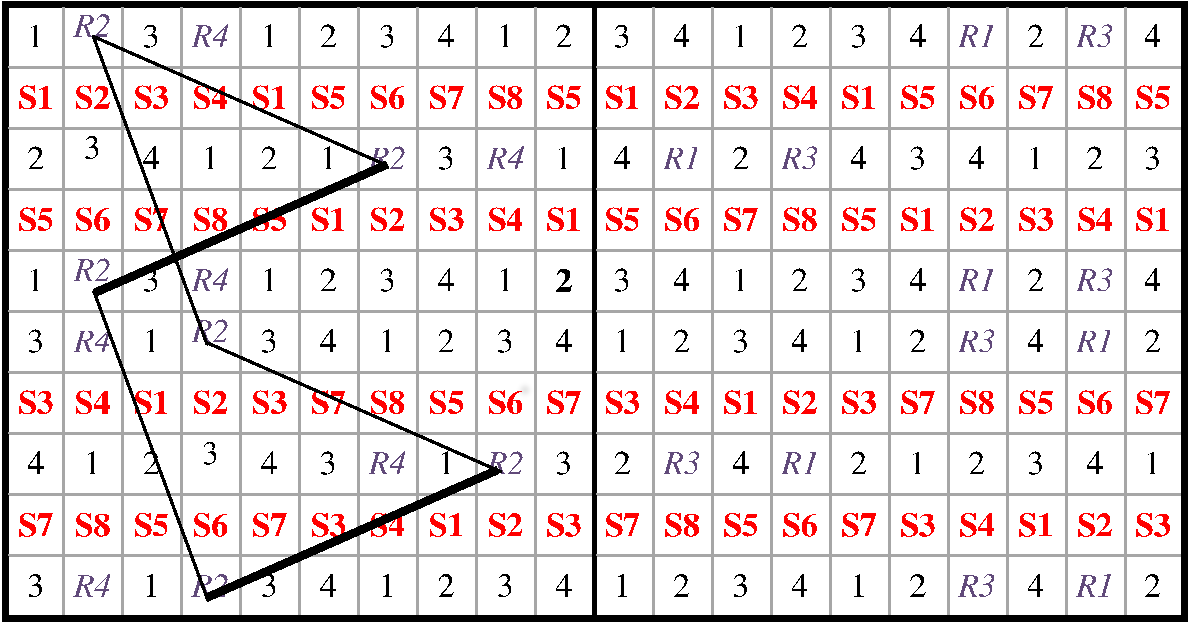}\\
						\small{Cycles $R1$ to $R4$.}
						\vspace{10pt}
						
						\includegraphics[scale=0.9]{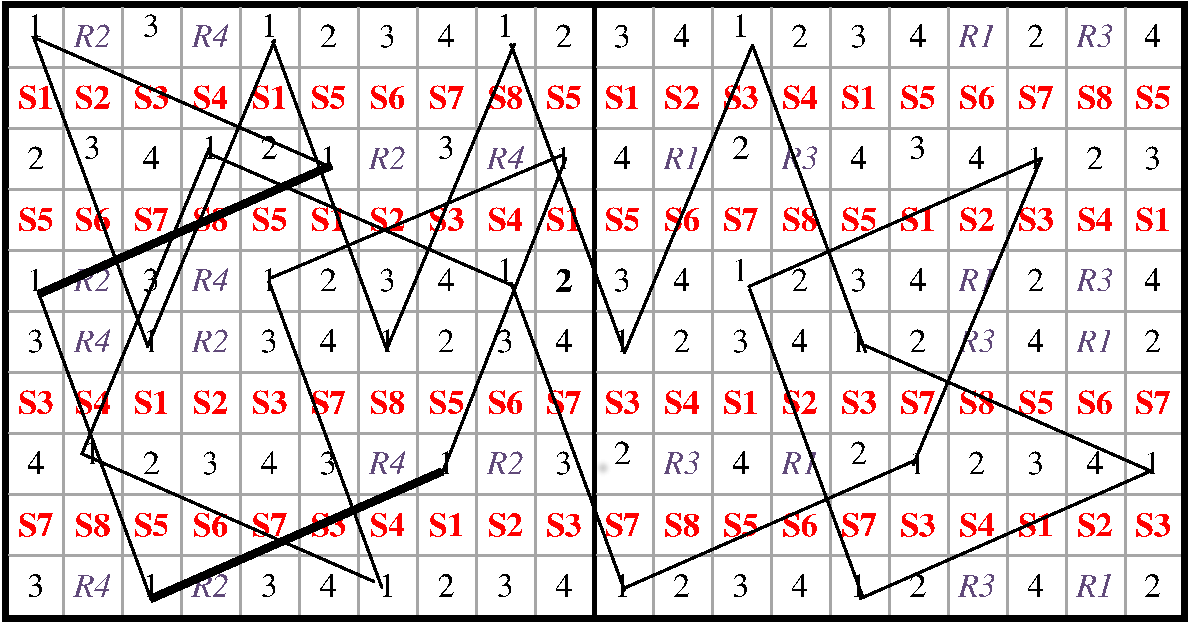}\\
						\vspace{3pt}
						\includegraphics[scale=0.9]{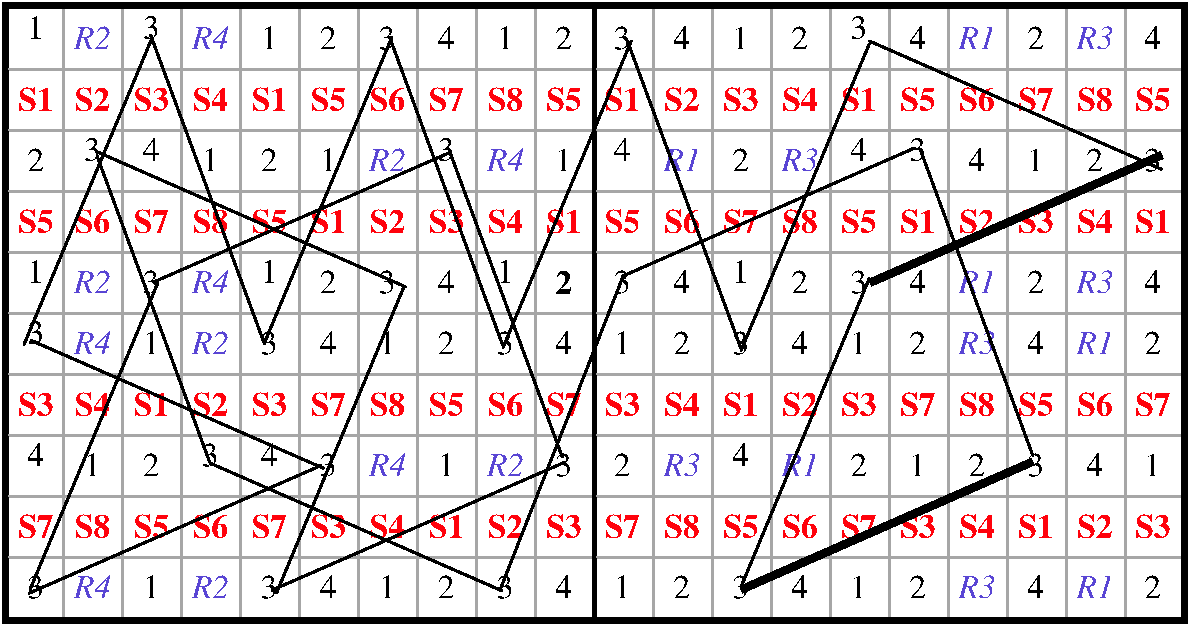}\\
						\small{Cycles 4 and 2 are mirror images of 1 and 3 respectively. They are longer and more complex because they were constructed as concatenations of smaller cycles (which look like $R1$-$R4$). }
					\end{center}
					
					\begin{lemma} \label{l14}
						For any $k \geq 2$, $[20] \times [10k]$ admits a $(2, 5)$ knight's tour.
					\end{lemma}
					\begin{proof}
						We divide the board into a sequence of  $k$ bricks. Cycle $x$ belonging to $i$-th block from the top will be denoted by $x^{(i) }$. The edges we use for concatenation are highlighted in figures above.
						
For each $i$, we first concatenate $S1^{(i) }$. and $S1^{(i+1) }$, $S2^{(i) }$ and $S2^{(i+1) }$ etc, as in the figure below. As usually, bridges are denoted by gray edges.

					\begin{center}
						\includegraphics[scale=0.9]{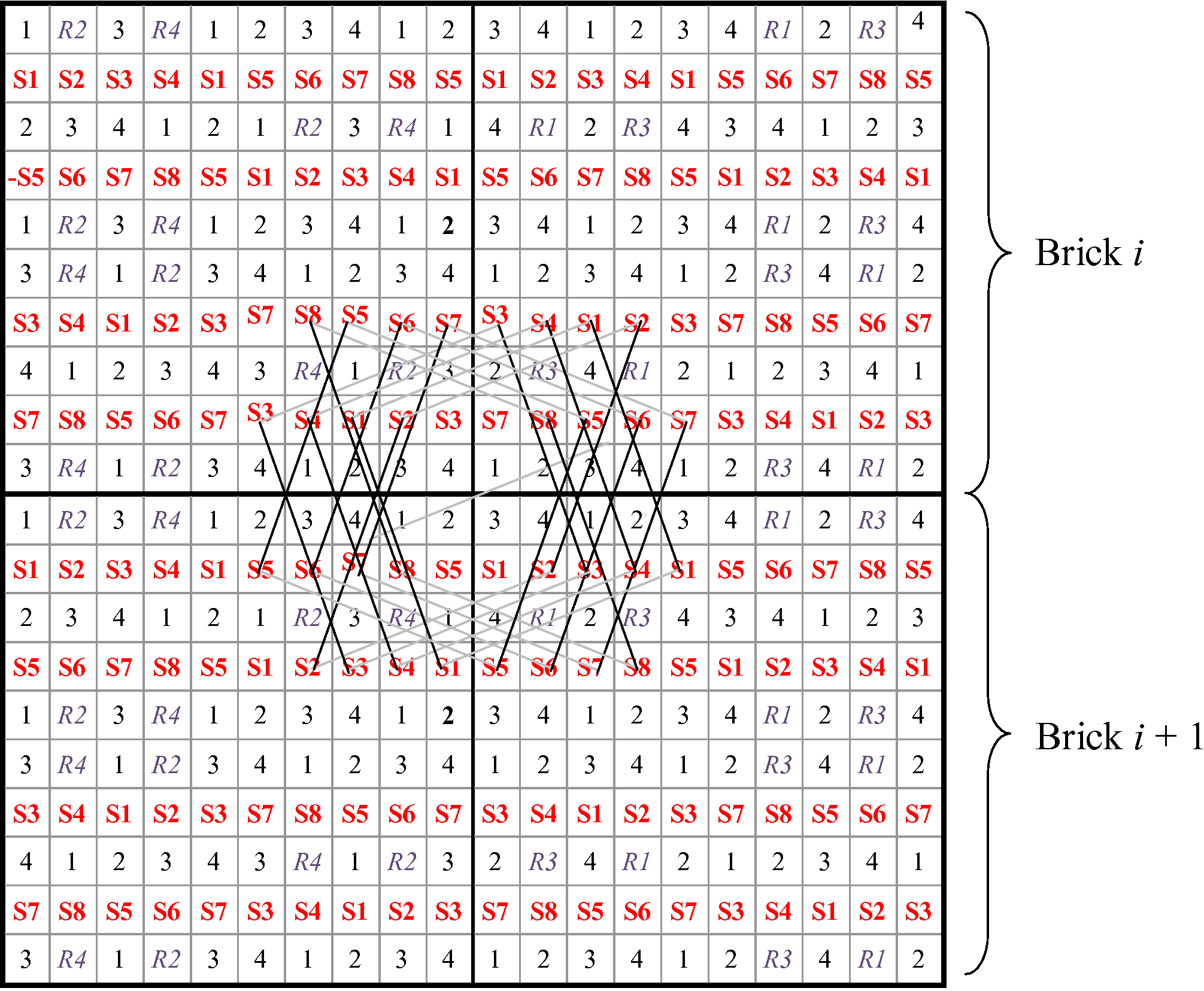}
					\end{center}
					
					We repeat the same with 1-4 and $R1$-$R4$.
					\begin{center}
						\includegraphics[scale=0.9]{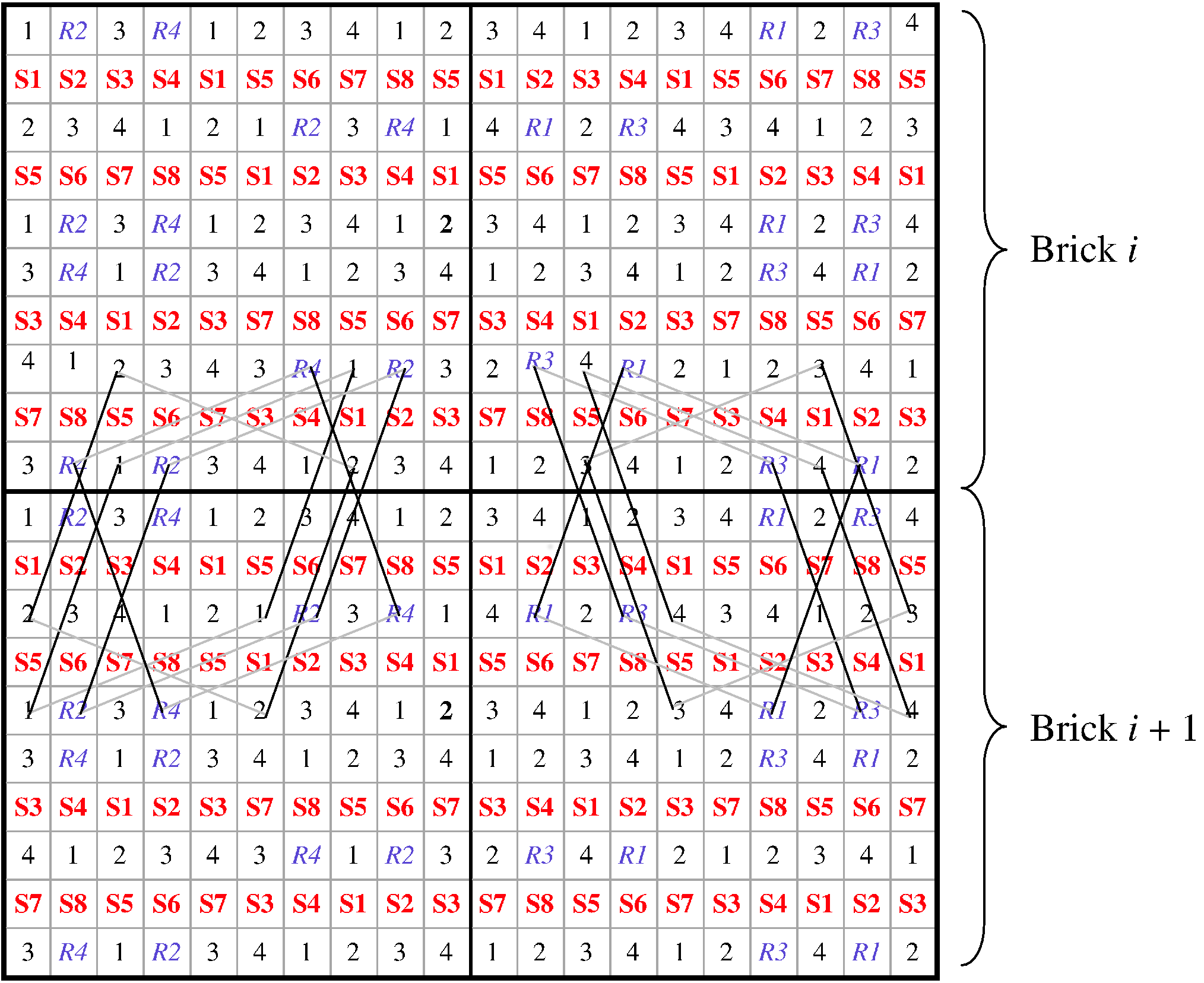}
					\end{center}
					Recall that these connections happen between each pair of neighbouring bricks. The result is a $20 \times 10k$ rectangle partitioned into 16 cycles: $S1$-$S8$, 1-4 and $R1$-$R4$.
					These are concatenated using only the first and second brick as follows.
					\begin{center}
						\includegraphics[scale=0.9]{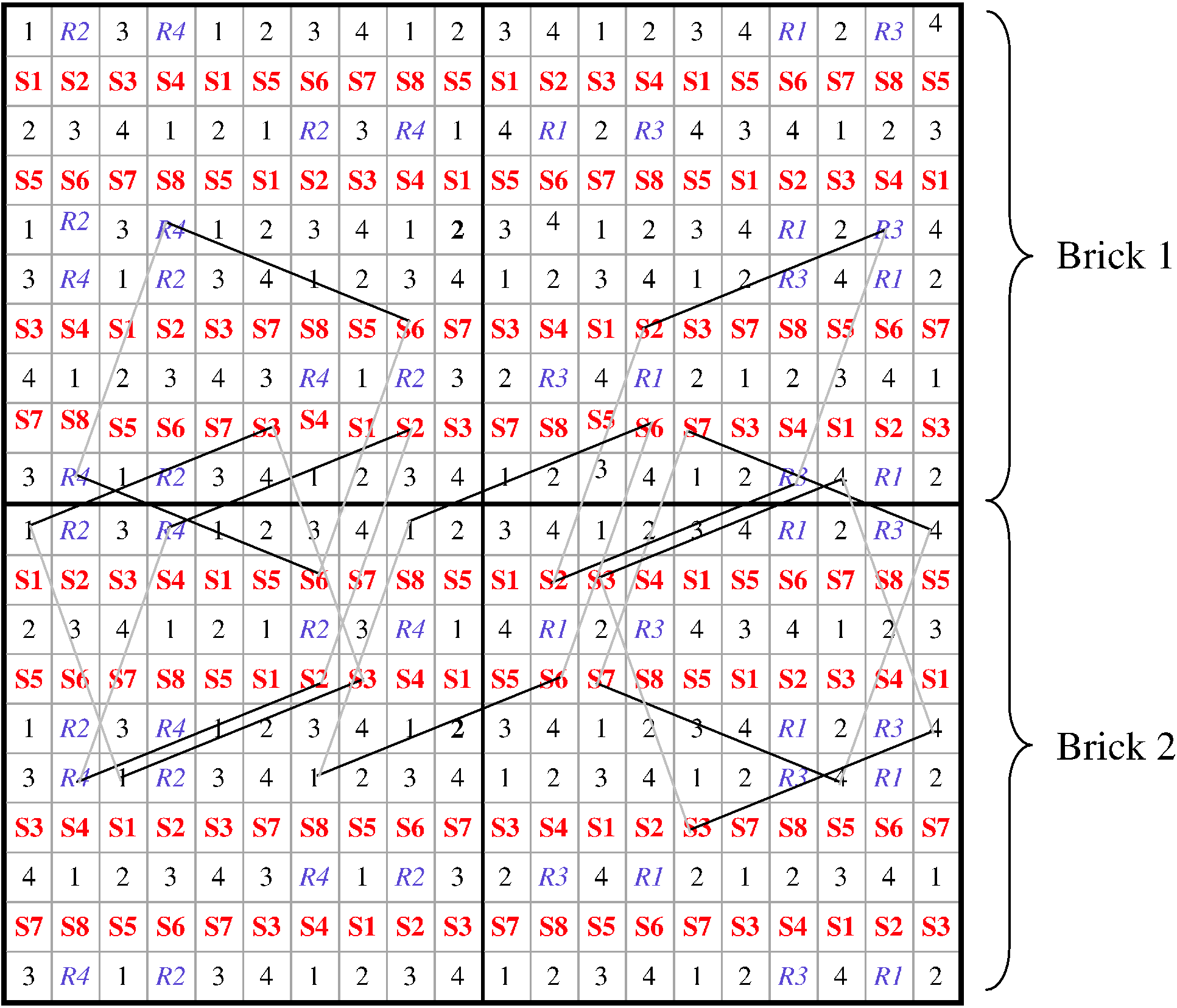}
					\end{center}
					This forms the cluster $S7 \cup 4 \cup S3 \cup 1 \cup S6 \cup R4 \cup S2 \cup R3$. Similarly, we form cluster $S1 \cup 2 \cup S5 \cup 3 \cup S4 \cup R2 \cup S8 \cup R1$.
					\begin{center}
						\includegraphics[scale=0.9]{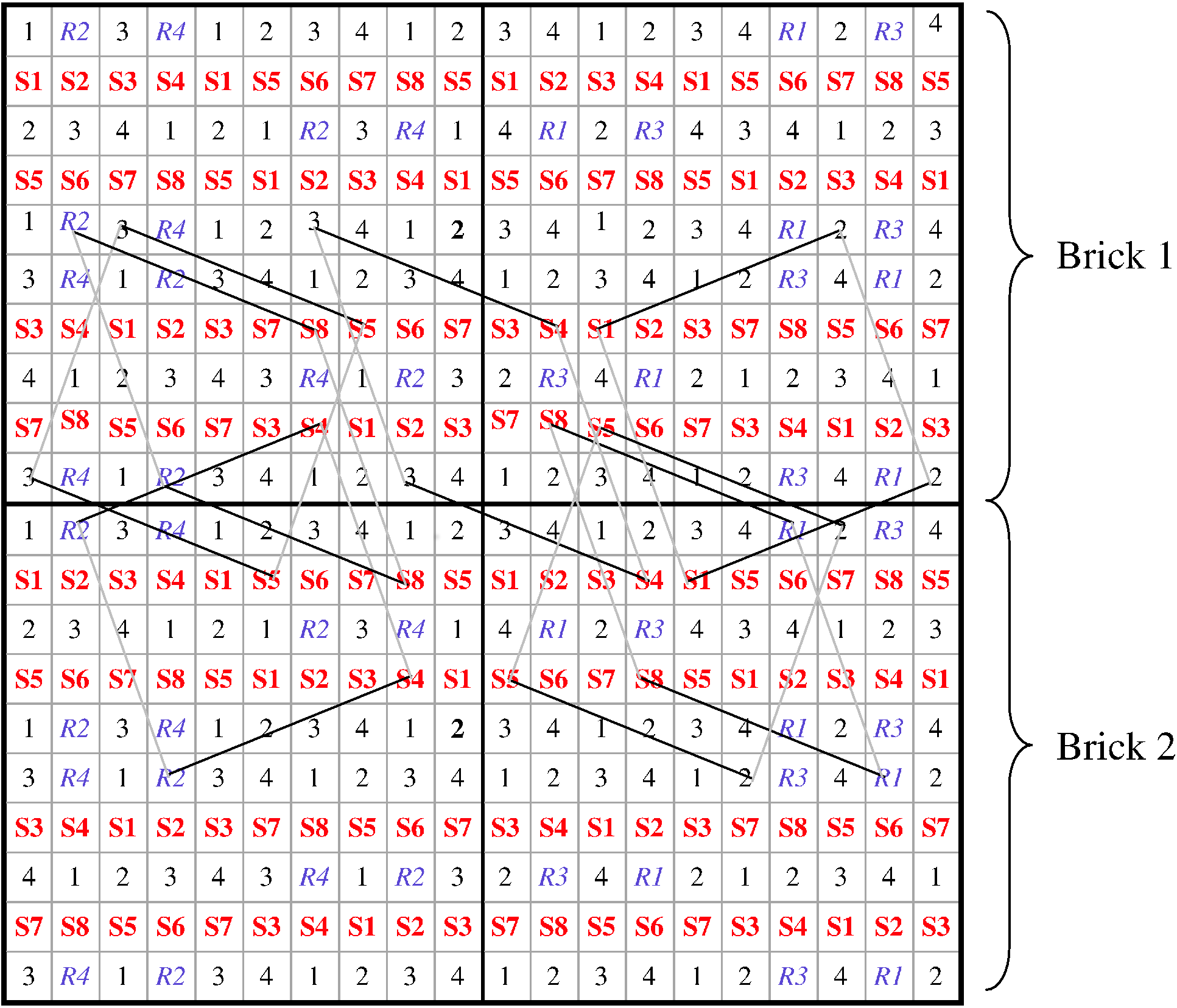}
					\end{center}
					Finally, the two clusters are concatenated using the following bridge.
					\begin{center}
						\includegraphics[scale=0.9]{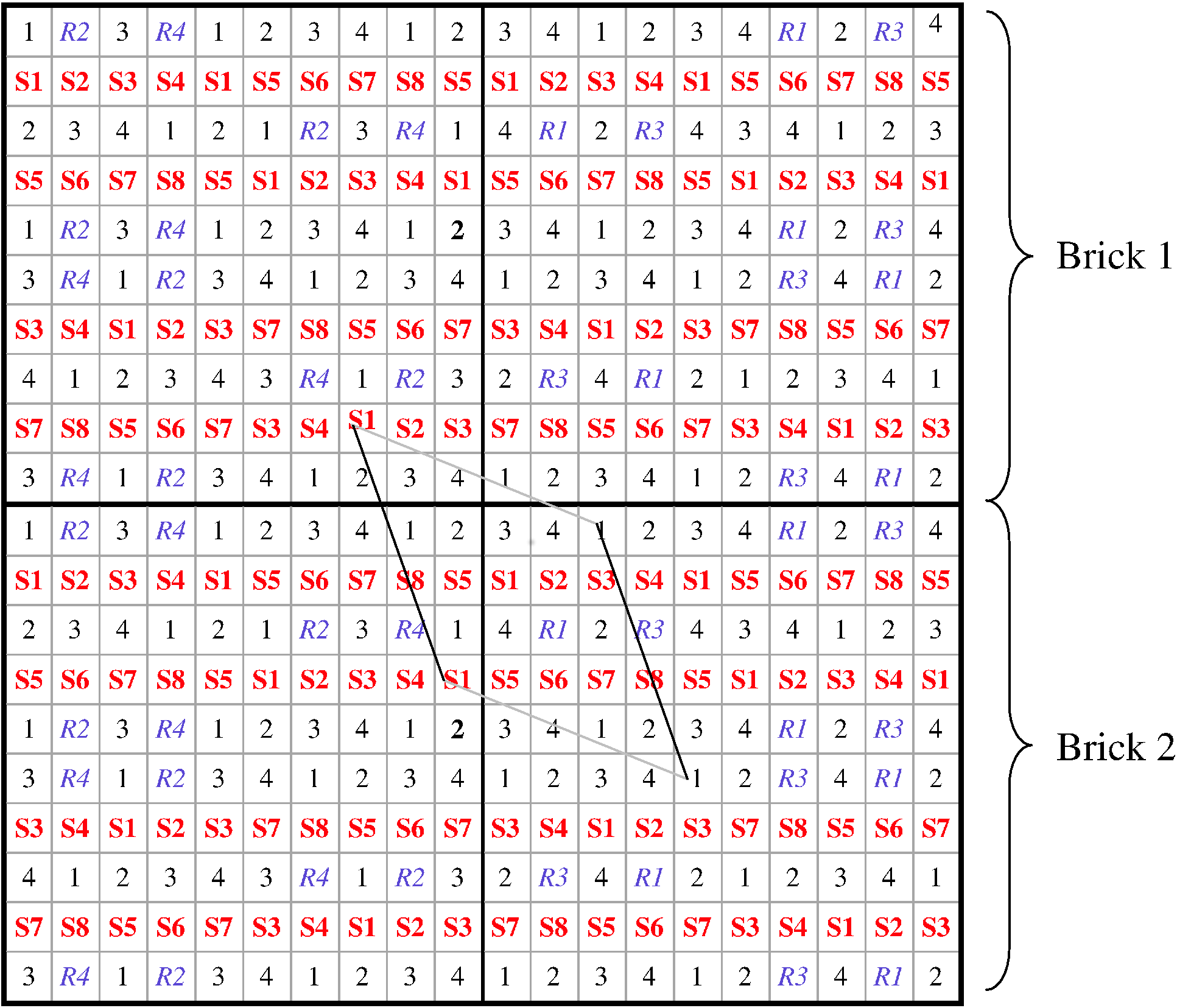}
					\end{center}
				\end{proof}
	
					The obvious purpose of this lemma is to get a $(2, 5)$ knight's tour in $[20] \times [20]$, but it is also crucial for constructing a rectangle, whose exact size is yet to be determined.
					
\subsubsection{Rectangle \texorpdfstring{$[20] \times [154]$}{[20] [154]}}
					As suggested by the subtitle, we insert a remainder of side length $4$ (shaded) into the construction in Lemma \ref{l14}.	
					
				\begin{wrapfigure}[13]{L}{0.17\textwidth}
    				\includegraphics[scale=0.31]{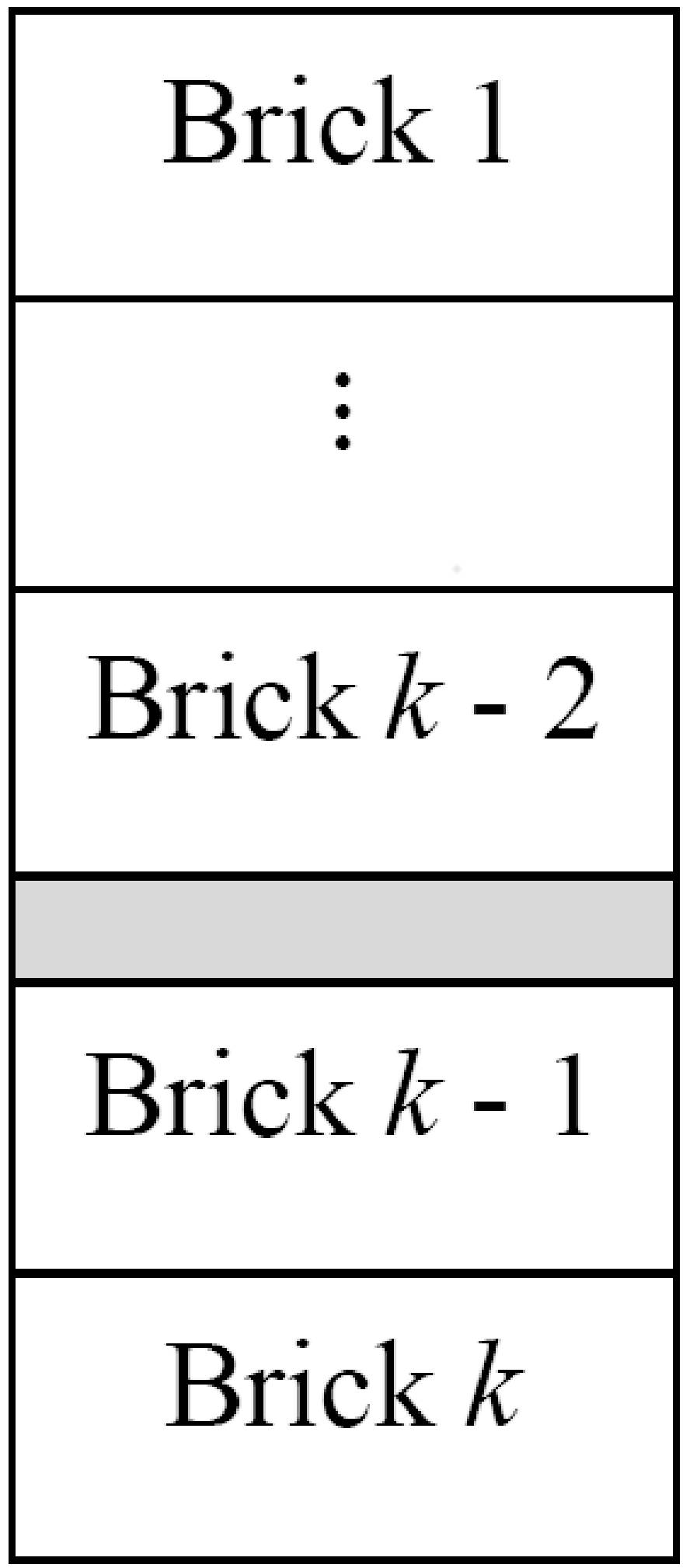}
				\end{wrapfigure}
				
Upper and lower rectangles of side lengths $20 \times 10(k-2)$ and  $20 \times 20$ are equipped with tours from Lemma \ref{l14}, but the bottom one is reflected vertically. We call these tours $U$ and $L$ respectively. The remainder is filled in by replacing the gray edges with black ones in Figure \ref{g27}. 
				
				\begin{lemma} \label{l15}
					For any $k \geq 4$, $[20] \times [10k+4]$ admits a $(2, 5)$ knight's tour.
				\end{lemma}
				\begin{proof}
					
			\begin{figure}[!ht]
				\centering
				\includegraphics{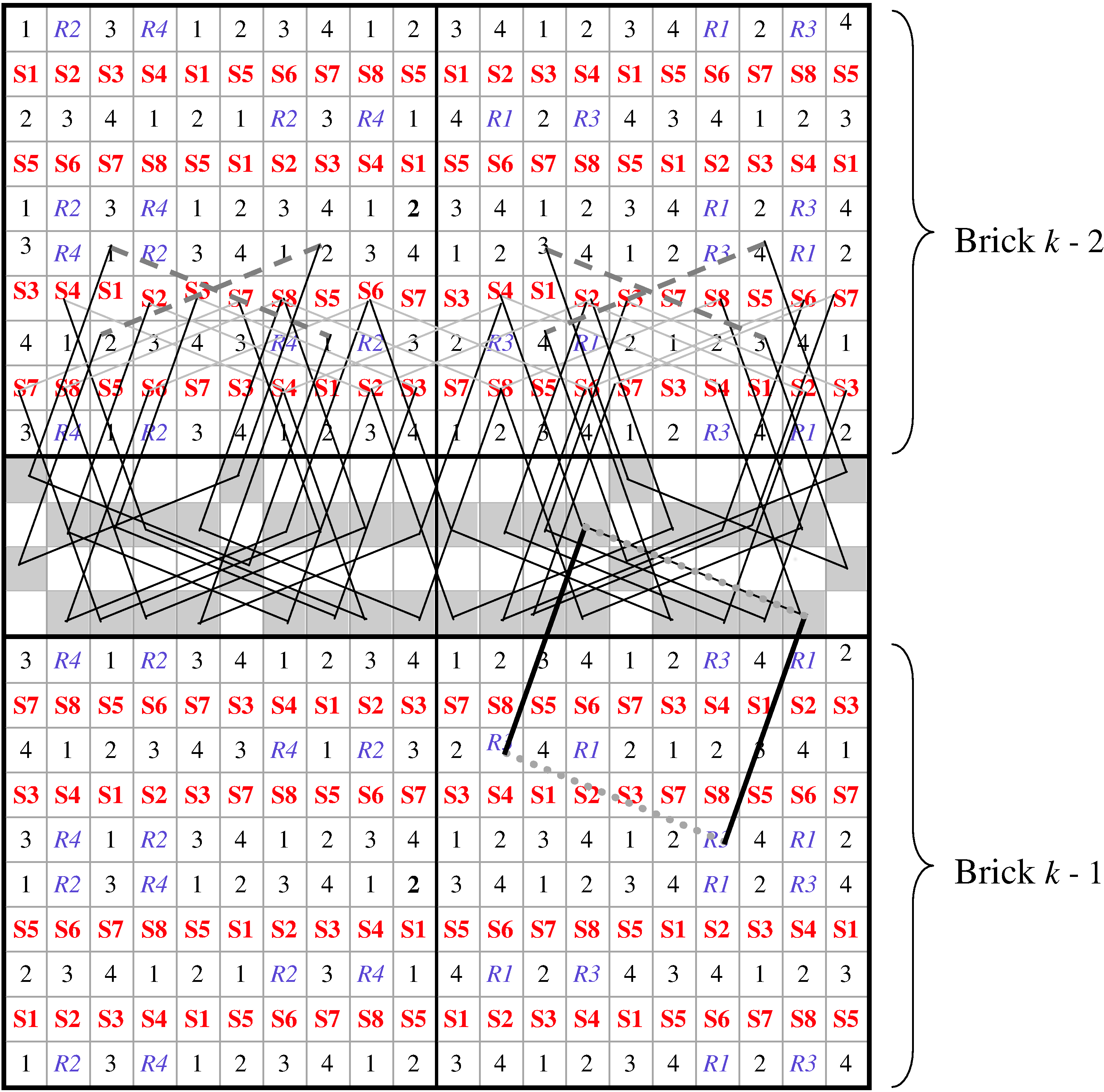}
				\caption{Extending the $(2, 5)$ knight's tours constructed in Lemma \ref{l14} to the remainder.
} \label{g27}
			\end{figure}
					
					Figure \ref{g27} shows how tour $U$ is extended to traverse the subset of the remainder indicated by gray shading.
					
					The complement of this subset is exactly its mirror image under reflection in the horizontal bisector of the remainder, so it can be filled in by extending tour $L$ in the same way.

					This leaves us with two separate cycles which originate from $U$ and $L$. These are concatenated using the bridge formed by dotted edges, so the proof is complete.
				\end{proof}
			\phantom{\large{some text to complete\\}}
			This gives us sufficient tools to deduce the main results of this section.	
			
				\begin{figure}[!ht]
					\centering
					\includegraphics[scale=0.8]{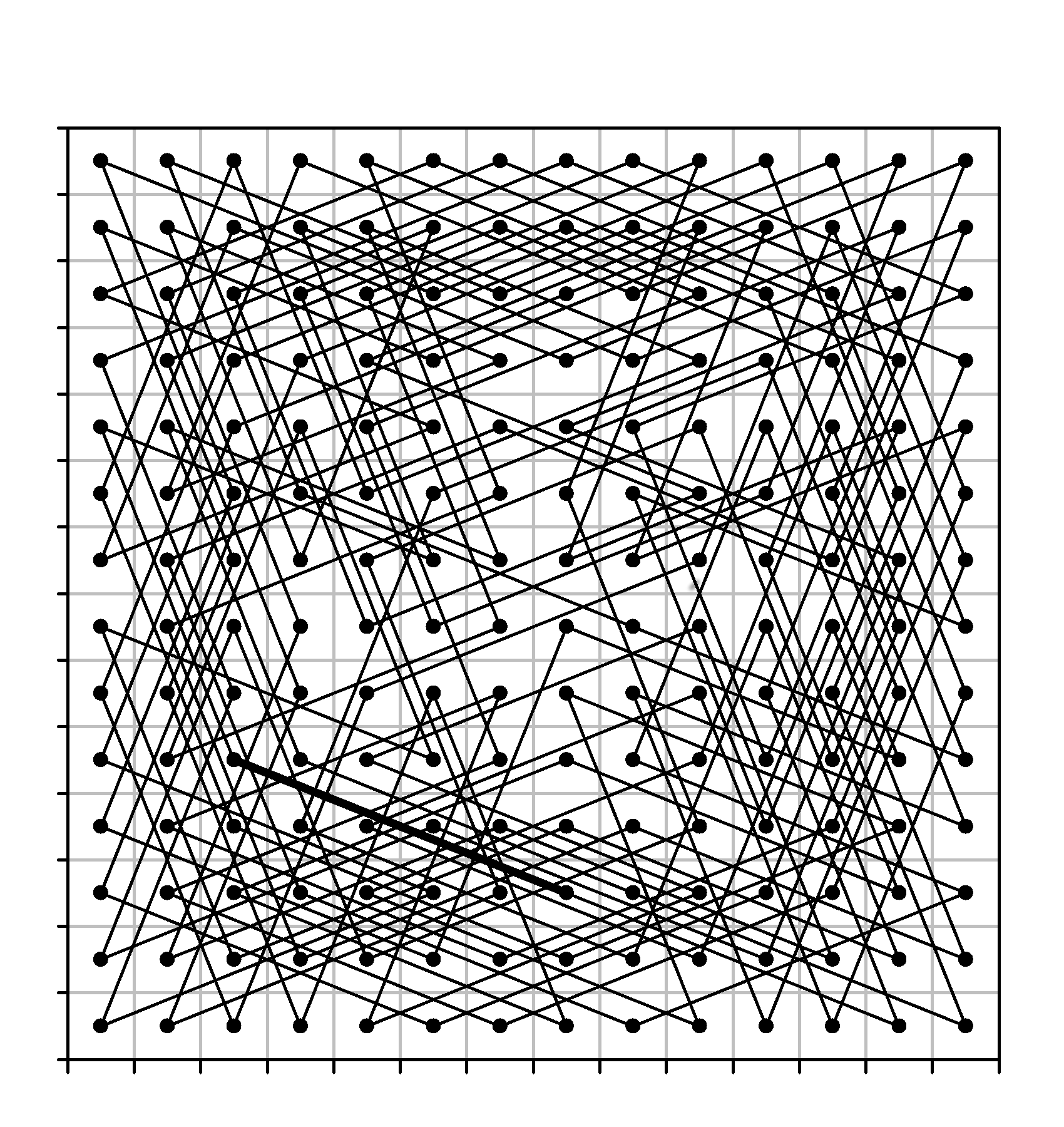}
					\caption{A $(2, 5)$ knight's tour in $[14]\times [14]$, constructed using one of our programs. Link H and link V are both given by the highlighted edge.
	} \label{gs2}
				\end{figure}
				
				\begin{theorem}
					For any sufficiently large even number $n$, there is a $(2, 5)$ knight's tour in $[n]^2$.
				\end{theorem}
				\begin{proof}
					We use Figure \ref{gs2} to construct a $(2, 5)$ knight's tour in a square of side length $154 = 11 \cdot 14$ as described in Lemma \ref{l13}. This tour contains both links.
					
					The $(2, 5)$ knight's tour in the $20 \times 154$ rectangle given by Lemma \ref{l13} contains both link H and link V as a part of cycle $3^{(1) }$ (thick edges):
					\newpage
					\begin{center}
						\includegraphics[scale=0.8]{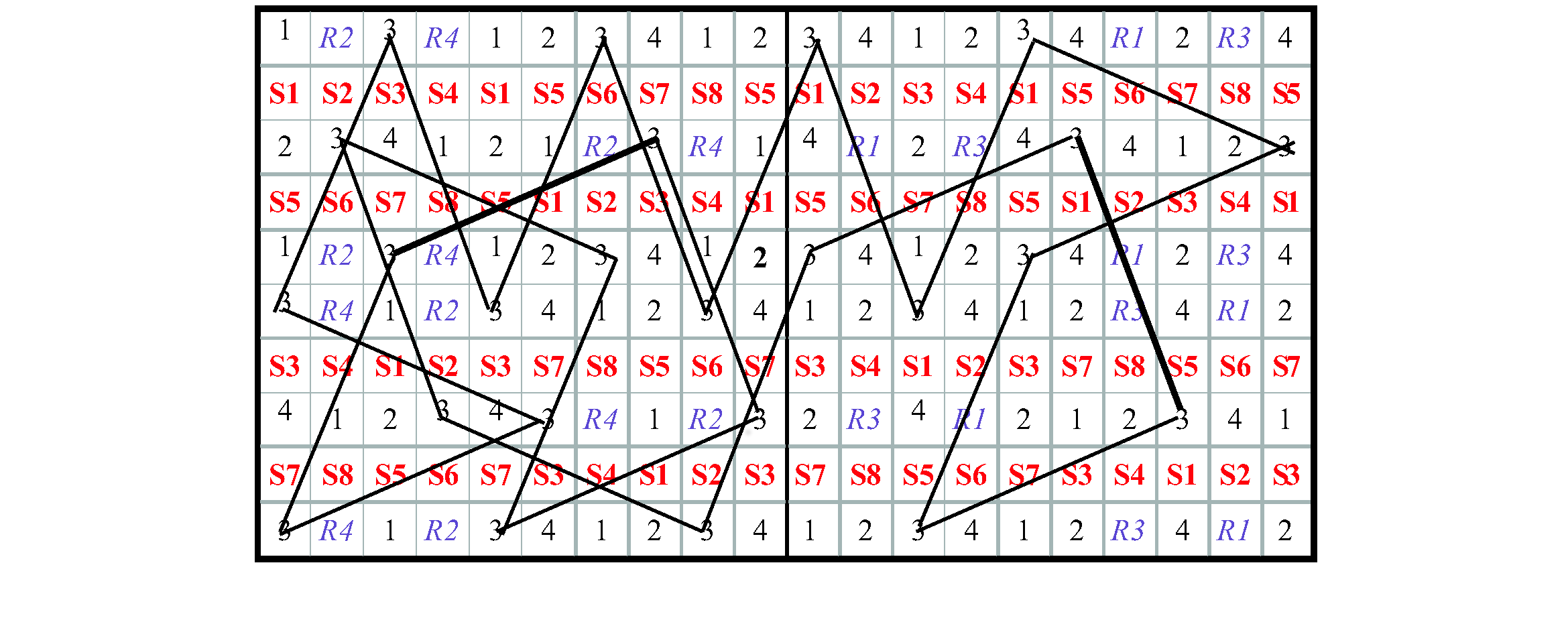}
					\end{center}
					Now we have rectangles of side lengths $20 \times 20$,  $20 \times 154$ and $154 \times 154$ containing both links, so the claim holds by Lemma \ref{l13}.
					
				\end{proof}
				
			\begin{corollary}
				For sufficiently large even $n$ and $d \geq 2$, there is a $(2, 5)$ knight's tour in $[n]^d$.
			\end{corollary}
			\begin{proof}
				Symmetric images of the edge  $\{(0, 2), (5, 0)\}$  exist in each corner of our brick (as part of cycles $1$, $2$, $3$ and $4$) and in each corner of our tour in the $14 \times 14$  square in Figure \ref{gs2}. Hence the same holds for our building blocks $154 \times 154$, $154 \times 20$ and $20 \times 20$ rectangles).
				
				The $(2, 5)$ knight's tour in $[n]^2$ is assembled from these blocks as described in Lemma \ref{l13}, so it contains link $\alpha$ and link $\beta$ regardless of the orientations of building blocks. The hypothesis of Lemma \ref{l9} is satisfied, so we deduce the result.
			\end{proof}
		\subsection{Further research}
				For the general case (with $a$ and $b$ still coprime), we can only notice that there is no apparent difference among $(a, b)$ knight's graphs that would imply or indicate non-existence of $(a, b)$ knight's tours. The following conjecture is already formulated by Erde, Gol\'{e}nia and Gol\'{e}nia.
			\begin{conjecture}[Erde, Gol\'{e}nia, Gol\'{e}nia \cite{erde}]  \label{c11}
				Let  $\gcd(a, b) = 1$, with $a$ and $b$ not both odd. For sufficiently large even $n$ and $d \geq 2$, there is an $(a, b)$ knight's tour in $[n]^d$.
			\end{conjecture}
				
				The question of necessary conditions for existence of an $(a, 1)$ knight's tour for even $a$ also remains open. An easy colouring argument can be used to show that the smallest connected knight's graph, on the board $[2a]^2$, is not Hamiltonian.
				
				We could try using the method from Section 2.2 to construct a knight's tour in $[4a+2]^2$ or even $[2a+2]^2$. Namely, choices have been made in constructing the guide and its implementation. It might be possible to modify our algorithm to work for $[4a+2]^2$ or even $[2a+2]^2$. The simplest two knights, (2, 1) and (4, 1), confirm this conjecture.
				
				However, new approaches are needed for constructing $(a, 1)$ knight's tours in $[n]^2$, for $2a+2 \leq n \leq (6a+2)a$.
	
	\appendix
	\section{Hamiltonian paths in grid graphs}
		Recall that the moves of an $(a, 1)$ knight constrained to one level are now reduced to just horizontal and vertical steps from one block to its neighbour. Thus in traversing each level (moves A1 - A3), we assumed the existence of certain Hamiltonian paths in $7 \times 7$, $7 \times 6$, $6 \times 6$. We now exhibit those paths.
		
		To avoid new notation, we still regard the vertices of our grid graph as blocks $B_{ij}$, and two blocks are adjacent if they share a side. We denote T$= \{B_{ij}: i, j \in \{2, 3\} \}$, T1 $= \{B_{ij}: i, j \in \{1, 2, 3, 4\} \} \setminus $T.
		\begin{lemma*} The grid graph with vertex set $\{ B_{ij}: i, j \in [7]\}$ admits a Hamiltonian path from any gray block in T1 to any gray block in T.
		\end{lemma*}
		
		\begin{proof}
			We start with paths finishing at $B_{22}$.
			\begin{center}
				\includegraphics[scale=.6]{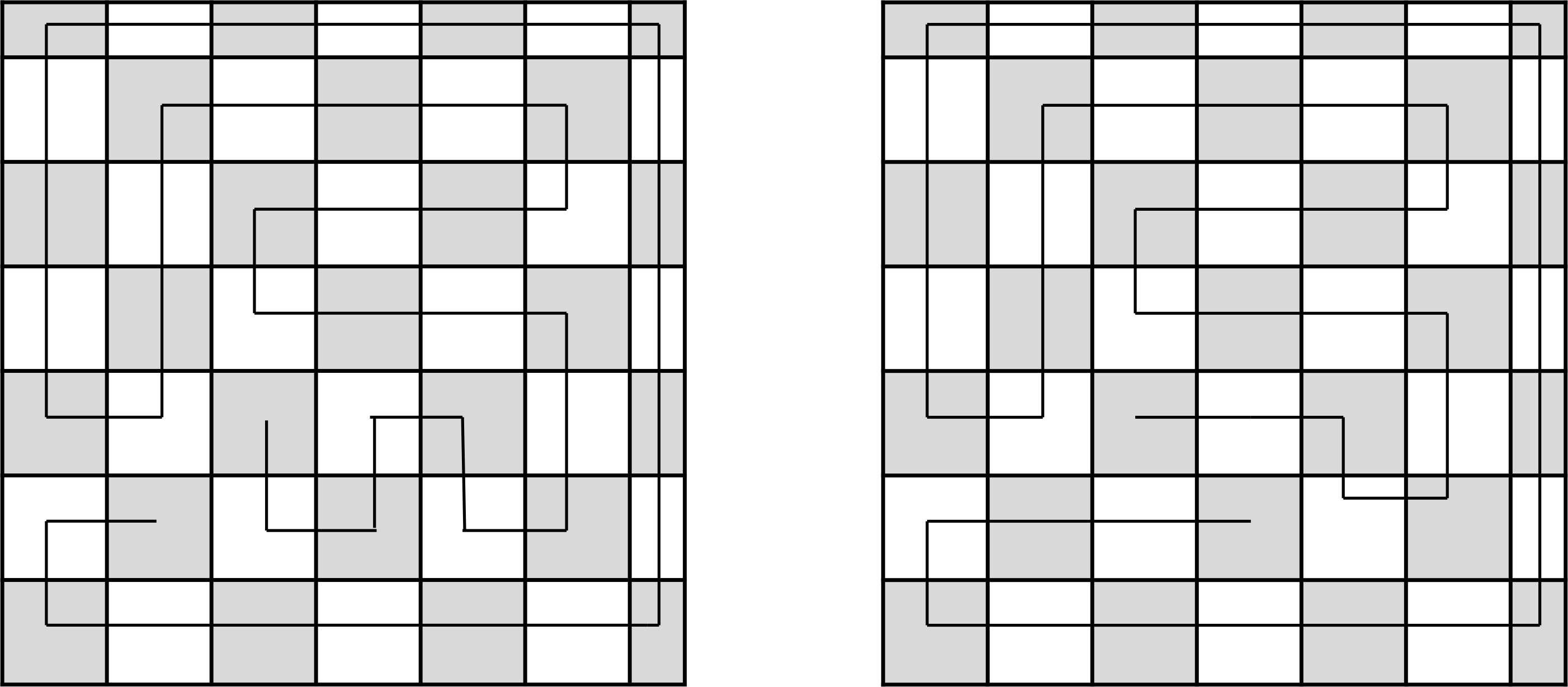}
				\vspace{10pt}
				\includegraphics[scale=.6]{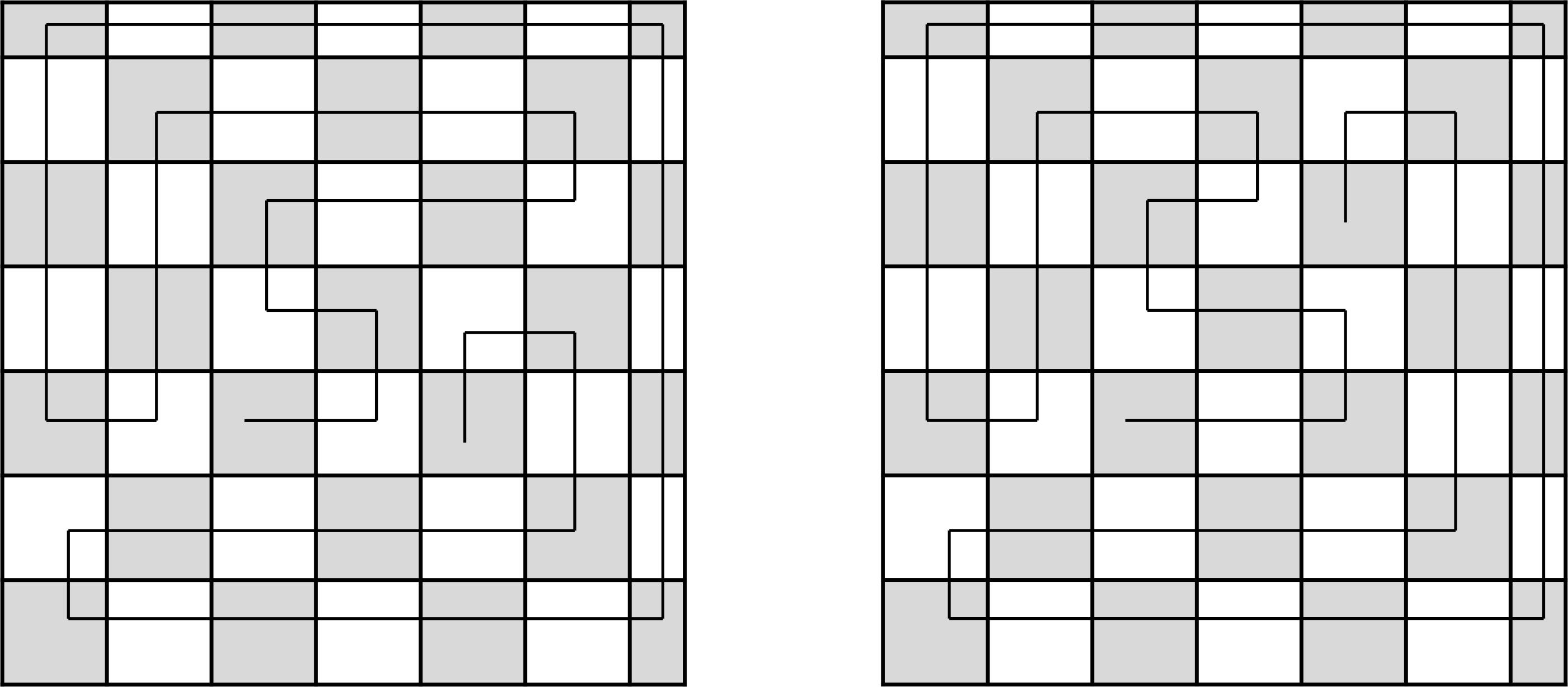}
			\end{center}
				Since $B_{22}$ is on the main diagonal, the paths starting at $B_{13}$ and $B_{24}$ are obtained using reflection across the main diagonal.
			
				By the same argument, it is sufficient to display the following paths finishing at $B_{33}$.
			
			\begin{center}
				\includegraphics[scale=.6]{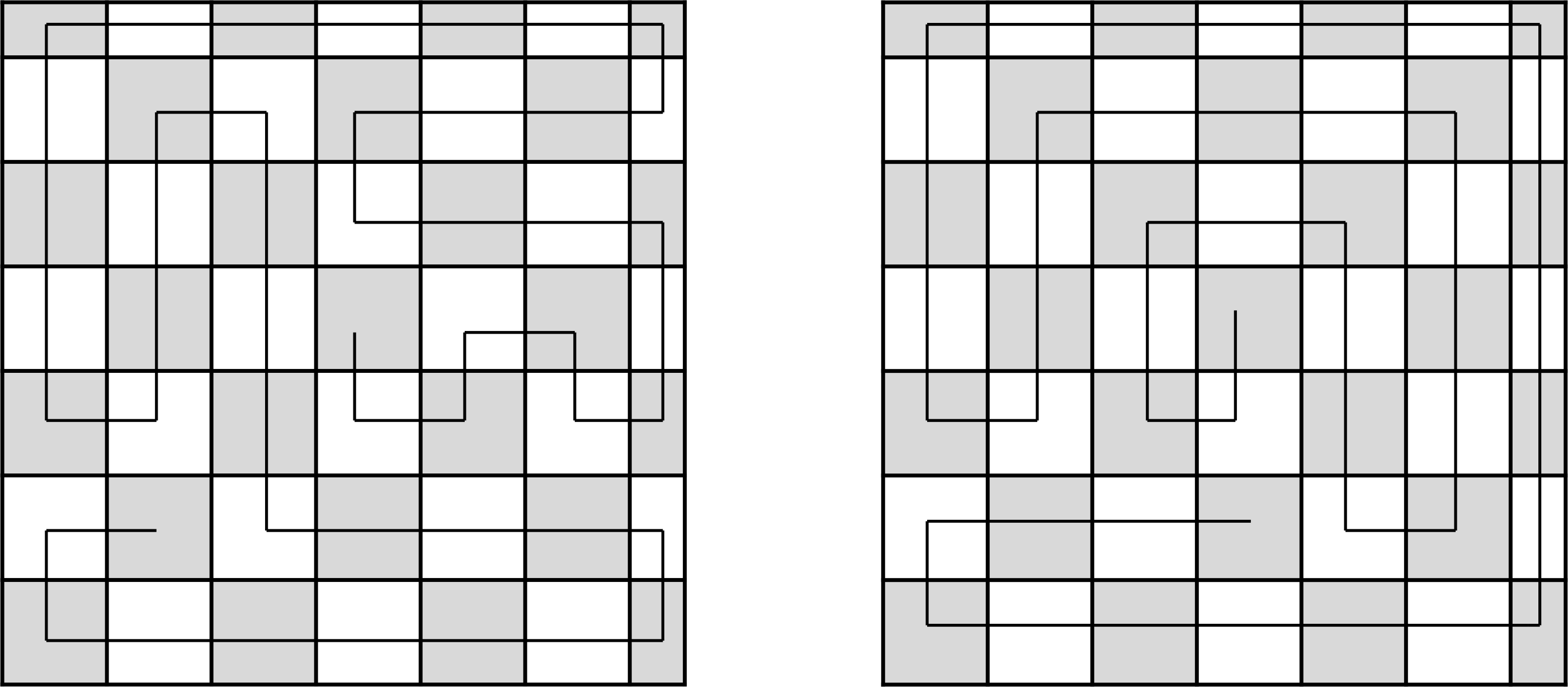}
				
				\vspace{10pt}
				
				\includegraphics[scale=.6]{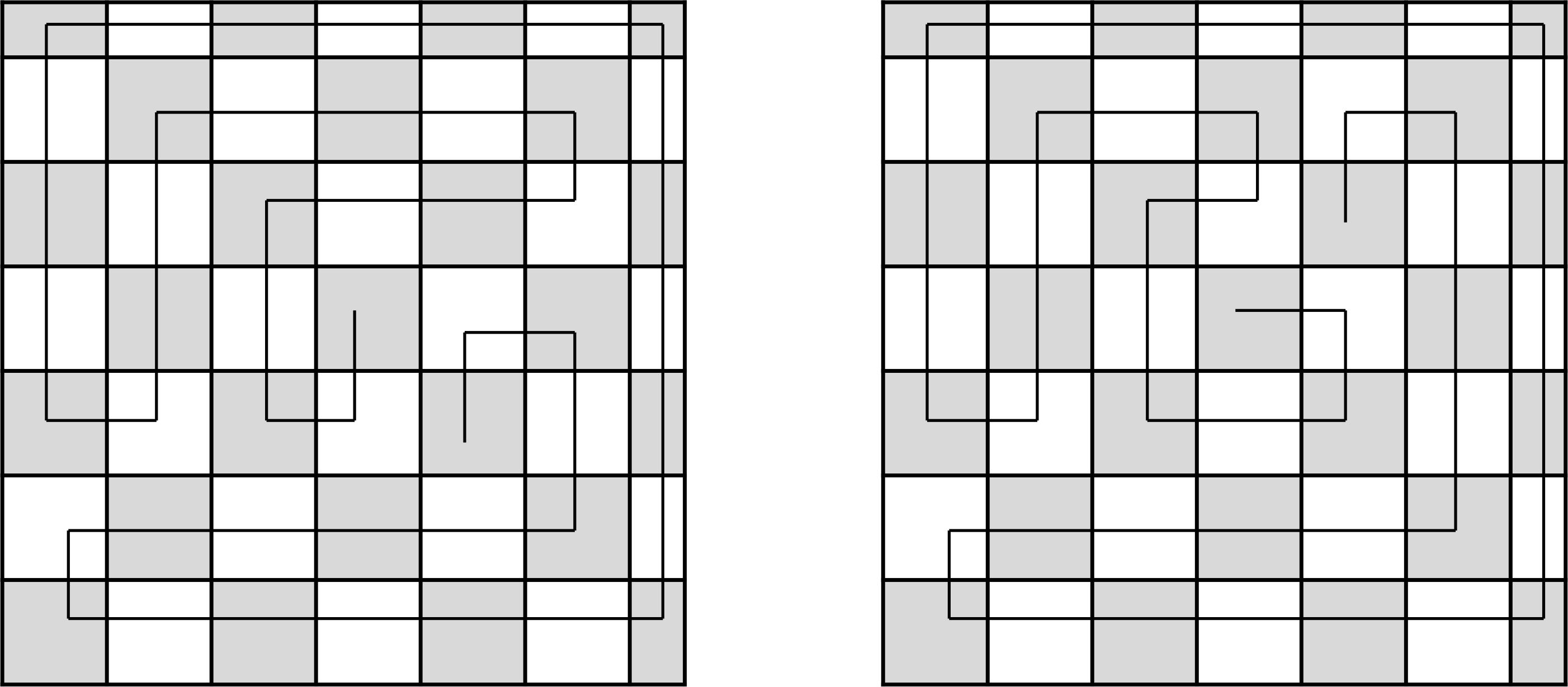}
			\end{center}
		\end{proof}
		
		\begin{lemma*} Let $H_{kl}$ be the grid graph with vertex set $\{ B_{ij}: (i, j) \in [k]\times [l] \}$.
		\vspace{-5pt} \begin{enumerate} \itemsep -1pt
			\item $H_{66}$ admits a Hamiltonian path from any gray block in T1 to any white block in T.
			\item $H_{66}$ admits a Hamiltonian path from any white block in T1 to any gray block in T.
			\item The same holds for $H_{67}$ and $H_{76}$.
		\end{enumerate}
		\end{lemma*}
		
		\begin{proof}
			We only need to exhibit three paths finishing at $B_{32}$ - the remaining three are constructed by reflecting across the antidiagonal.
			\begin{center}
				\includegraphics[scale=.6]{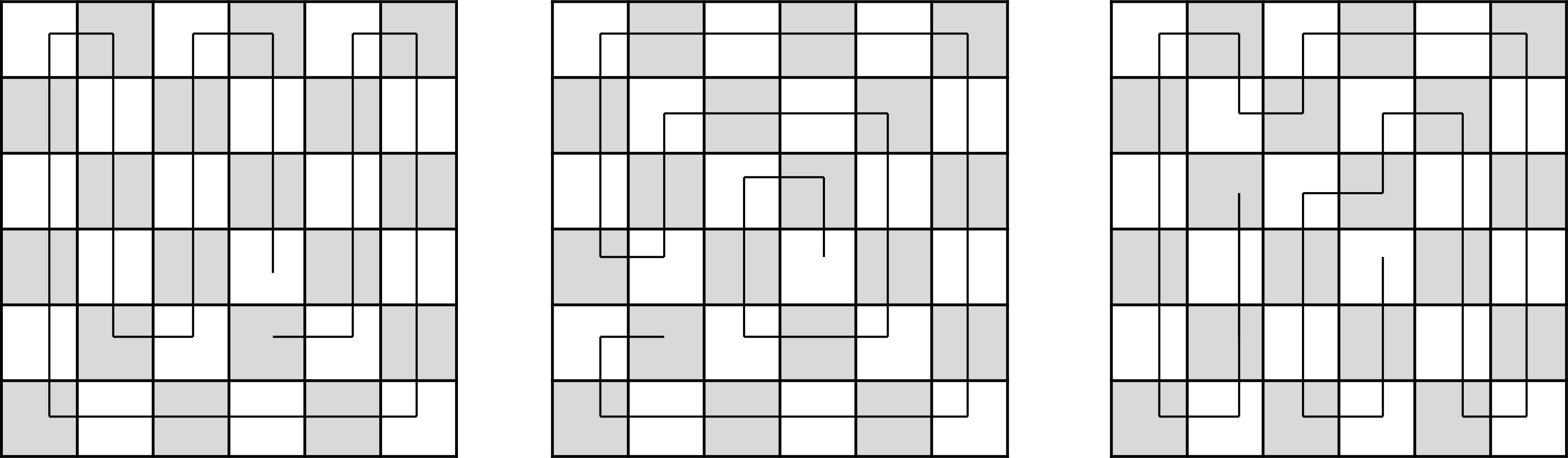}
			\end{center}
			Paths finishing at $B_{23}$ are obtained by reflecting across the main diagonal, which proves (i).
			
			For (ii), we use the following Hamiltonian paths and reflect them across the main diagonal to cover all the possible starting points (white blocks in T1), since $B_{22}$ is on the main diagonal. Reflection across the main diagonal to switch the finish point to $B_{33}$.
			\begin{center}
				\includegraphics[scale=.6]{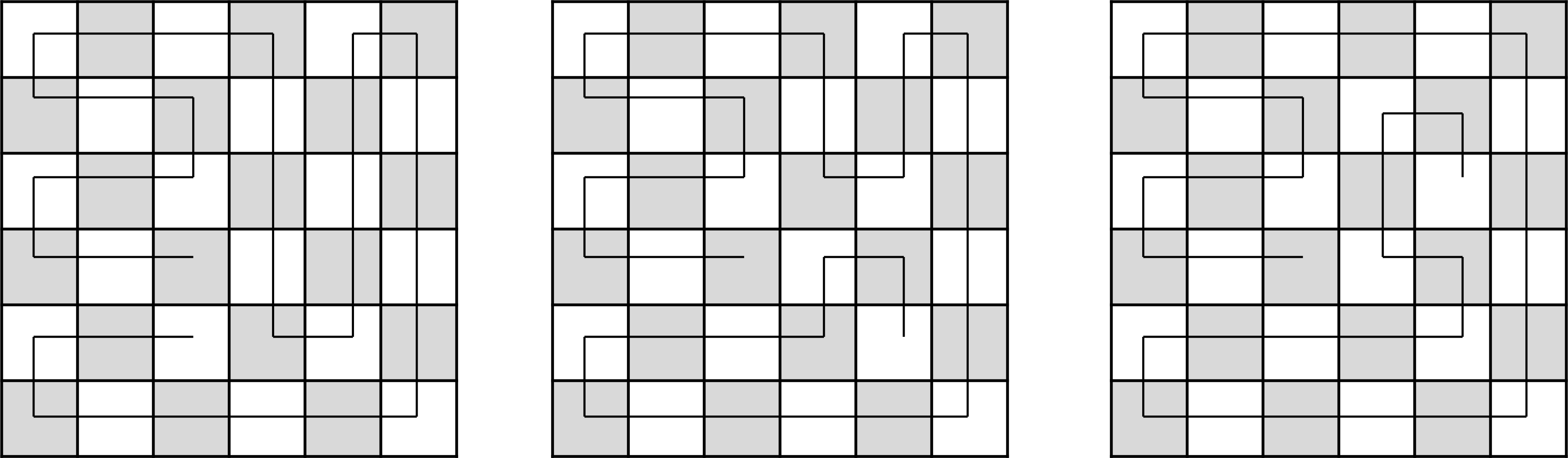}
			\end{center}
			
			Finally, the paths we constructed are easy to extend by a single row or column in any direction, so (iii) holds.
				Since $B_{22}$ is on the main diagonal, the paths starting at $B_{13}$ and $B_{24}$ are obtained using reflection across the main diagonal.
				
		\end{proof}
	\section*{Acknowledgements}
		I owe a debt of gratitude to Prof. Imre Leader for some very useful conversations.
		
		Financial support was provided by Trinity College, Cambridge, for which I am also very grateful.

\end{document}